\documentclass[11pt]{amsart}
\usepackage{graphicx,subcaption}
\usepackage{amsmath} 
\usepackage{amsthm,amsfonts,amssymb,mathrsfs,amscd,amstext,amsbsy} 
\usepackage{epic,eepic} 
\usepackage{yfonts}
\usepackage{paralist,enumerate}
\usepackage[all]{xy}
\usepackage{hyperref}
\hypersetup{colorlinks}
\usepackage{tikz}
\usetikzlibrary {positioning}
\definecolor {processblue}{cmyk}{0.96,0,0,0}
\usetikzlibrary{arrows}
\tikzset{    vertex/.style={circle,draw,minimum size=1.5em},    edge/.style={->,> = latex'}}

\newtheorem{theorem}{Theorem}[section]
\newtheorem{cor}[theorem]{Corollary}
\newtheorem{lemma}[theorem]{Lemma}

\newtheorem{theo}[theorem]{Theorem}
\newtheorem{lem}[theorem]{Lemma}
\newtheorem{pro}[theorem]{Proposition}

\newtheorem{exa}[theorem]{Example}

\newtheorem{Definition}[theorem]{Definition}
\newtheorem*{Definition*}{Definition}

\def\qed{\hfill \ifhmode\unskip\nobreak\fi\quad\ifmmode\Box\else$\Box$\fi\\ }

\begin{document}

\title[Four dimensional almost complex torus manifolds]{Four dimensional almost complex torus manifolds}
\author{Donghoon Jang}
\thanks{2020 Mathematics Subject Classifications: 57M60, 58D19, 58C30, 32Q60, 57S12, 14M25, 32M05}
\thanks{Key words: almost complex torus manifold, multi-fan, graph, blow up, blow down}
\thanks{Donghoon Jang was supported by the National Research Foundation of Korea(NRF) grant funded by the Korea government(MSIT) (2021R1C1C1004158).}
\address{Department of Mathematics, Pusan National University, Pusan, South Korea, 46241}
\email{donghoonjang@pusan.ac.kr}

\begin{abstract}
An (almost) complex torus manifold is a $2n$-dimensional compact connected (almost) complex manifold equipped with an effective action of a real $n$-dimensional torus $T^n$ that has fixed points. Let $M$ be a 4-dimensional almost complex torus manifold. To $M$, we associate two equivalent combinatorial objects, a family $\Delta$ of multi-fans and a graph $\Gamma$, which encode the data on the fixed point set. We find a necessary and sufficient condition for each of $\Delta$ and $\Gamma$, and provide a minimal model and operations for each of $\Delta$ and $\Gamma$, which correspond to blow up and down of $M$. We discuss applications, including ones to complex manifolds.




\end{abstract}

\maketitle

\tableofcontents

\section{Introduction and statements of results} \label{s1}

In the theory of toric varieties, there is a correspondence between toric varieties and combinatorial objects, called fans.
In dimension 4, we extend the correspondence between compact nonsingular toric varieties and regular fans to a correspondence between almost complex torus manifolds and families of multi-fans in a geometric way. Moreover, we provide a minimal model and operations for a family of multi-fans associated to a 4-dimensional almost complex torus manifold. Furthermore, we discuss various applications.

An \textbf{(almost) complex torus manifold} is a $2n$-dimensional compact connected (almost) complex manifold equipped with an effective action of a real $n$-dimensional torus $T^n$ that has fixed points. 

We will now explain the main results of this paper in slightly more detail. Let $M$ be a 4-dimensional almost complex torus manifold.
\begin{enumerate}[(1)]
\item (Subsection \ref{s1.2}) To $M$, we associate two equivalent combinatorial objects, a family $\Delta$ of multi-fans and a (directed labeled) graph $\Gamma$. These each contain information about the data on the fixed point set of $M$, the weights at the fixed points and the isotropy spheres connecting fixed points.
\item (Subsection \ref{s1.2}) We present a necessary and sufficient condition for each of $\Delta$ and $\Gamma$ (Theorem \ref{tt19}).
\item (Subsection \ref{s1.3}) We provide a minimal model and operations for each of $\Delta$ and $\Gamma$. We introduce two operations on $\Delta$ and $\Gamma$ that correspond to equivariant blow up and down of $M$.
 Theorem \ref{tt114} states that we can equivariantly blow up and down $M$ ($\Delta$ and $\Gamma$) to another 4-dimensional almost complex torus manifold $M'$ (another family $\Delta'$ of multi-fans and another graph $\Gamma'$), which is minimal in the sense that every weight at a fixed point of $M'$ is a unit vector in $\mathbb{Z}^2$ ($\Delta'$ has only unit vectors and the label of each edge of $\Gamma'$ is a unit vector.)
\item (Subsection \ref{s1.3}) Theorem \ref{tt114a} establishes the existence result that for any family $\Delta$ of multi-fans (or a graph $\Gamma$) obtained from our minimal model and operations, there exists a 4-dimensional almost complex torus manifold whose fixed point data is contained in $\Delta$ ($\Gamma$).
\item (Subsection \ref{s1.4}) We discuss applications of our results. When $M$ is complex, we can blow up and down $M$ to $\mathbb{CP}^1 \times \mathbb{CP}^1$ (Theorem \ref{tt115}). Consequently, it follows that any two 4-dimensional complex torus manifolds are obtained from each other by blow up and down (Corollary \ref{tc116}). This implies the known fact that any two rational surfaces are obtained from each other by blow up and down. Additionally, our fan determines a 4-dimensional complex torus manifold up to equivariant biholomorphism (Corollary \ref{tc117}). Also, when $M$ is complex, $\Gamma$ is a GKM(Goresky-Kottwitz-MacPherson) graph if we forget the direction of each edge, and hence $\Gamma$ encodes the equivariant cohomology of $M$ (Corollary \ref{c120}).
\end{enumerate}

\subsection{Preliminaries} \label{s1.1}

An \textbf{almost complex manifold} $(M,J)$ is a manifold $M$ with a smooth bundle map $J:TM \to TM$ that restricts to a linear complex structure on each tangent space, called an \textbf{almost complex structure}. An action of a group $G$ on an (almost) complex manifold $(M,J)$ is said to \textbf{preserve the (almost) complex structure} if $dg \circ J=J \circ dg$ for all $g \in G$. Throughout this paper, we assume that any group action on an (almost) complex manifold preserves a given (almost) complex structure.

For an action of a group $G$ on a manifold $M$, we denote its fixed point set as $M^G$ its fixed point set, which is the set of points in $M$ fixed by the $G$-action. That is,
\begin{center}
$M^G=\{m \in M \, | \, g \cdot m = m \textrm{ for all } g \in G\}$.
\end{center}

Let a $k$-dimensional torus $T^k \cong (S^1)^k$ act on a $2n$-dimensional almost complex manifold $M$. Let $F$ be a fixed component. Let $\dim F=2m$. For $p \in F$, the normal space $N_pF$ of $F$ at $p$ decomposes into the sum of $n-m$ complex 1-dimensional vector spaces
\begin{center}
$N_pF=L_{F,1} \oplus \cdots \oplus L_{F,n-m}$,
\end{center}
where on each $L_{F,i}$ the torus $T^k$ acts by multiplication by $g^{w_{F,i}}$ for all $g \in T^k$, for some non-zero $w_{F,i} \in \mathbb{Z}^k$, $1 \leq i \leq n-m$. Here, for $g=(g_1,\cdots,g_k) \in T^k \cong (S^1)^k \subset \mathbb{C}^n$ and $w=(w_1,\cdots,w_k) \in \mathbb{Z}^k$, by $g^w$ we mean $g^w:=g_1^{w_1} \cdots g_k^{w_k}$. These elements $w_{F.1}$, $\cdots$, $w_{F,n-m}$ are the same for all $p \in F$ and are called \textbf{weights} of $F$.

\subsection{Association and necessary and sufficient condition} \label{s1.2}

In this subsection, to a 4-dimensional almost complex torus manifold, we associate a family of multi-fans and a graph, and present a necessary and sufficient condition for each object. First, we associate a family of multi-fans.

\begin{Definition} \label{td11}
A (2-dimensional) \textbf{multi-fan} $V$ is a sequence of non-zero vectors $v_1$, $v_2$, $\cdots$, $v_k$ in $\mathbb{Z}^2$ such that $v_1$, $v_2$, $\cdots$, $v_k$, $v_1$ are in counterclockwise order or in clockwise order. 
For a multi-fan $V$, we define $T(V)$ to be the number of revolutions made by the sequence $v_1, v_2, \ldots, v_k, v_1$ around the origin of $\mathbb{R}^2$ and call it the \textbf{winding number} of $V$. A multi-fan $V=\{v_1,v_2,\cdots,v_k\}$ is a \textbf{fan} if $T(V)=1$. 
\end{Definition}

\begin{Definition} \label{td12}
We say that a multi-fan $V=\{v_1,v_2,\cdots,v_k\}$ is \textbf{admissible} if the following hold:
\begin{enumerate}[(1)]
\item For each $1 \leq i \leq k$, $v_{i-1}$ and $v_i$ form a basis of $\mathbb{Z}^2$ ($v_0:=v_k$).
\item $v_{i+1}=-a_i v_i-v_{i-1}$ for some integer $a_i$, $1 \leq i \leq k$ ($v_0:=v_k$ and $v_{k+1}:=v_1$).
\end{enumerate}
\end{Definition}

\begin{Definition} \label{tsphere}
Let a $k$-dimensional torus $T^k$ act on an almost complex manifold $M$. Let $p$ and $q$ be fixed points and let $w \in \mathbb{Z}^k$. We say that an ordered pair $(p,q)$ is a \textbf{w-sphere}, if $M$ has an almost complex submanifold $\mathbb{CP}^1$ where the $T^k$-action on $M$ restricts to act by 
$$g \cdot [z_0:z_1]=[z_0:g^{w} z_1],$$
such that the $T^k$-action on this $\mathbb{CP}^1$ has fixed points $p=[1:0]$ and $q=[0:1]$ with weights $w$ and $-w$ for this action, respectively. 
\end{Definition}

\begin{Definition} \label{td13}
Let $M$ be a 4-dimensional almost complex torus manifold. We say that a family of multi-fans $V_1,V_2,\cdots,V_m$  \textbf{describes} $M$ if we can partition its fixed point set $M^{T^2}$ into disjoint sets $A_1$, $A_2$, $\cdots$, $A_m$, so that for each $A_j=\{p_{j,1},\cdots,p_{j,k_j}\}$, $(p_{j,i},p_{j,i+1})$ is $v_{j,i}$-sphere, for $1 \leq i \leq k_j$ and $1 \leq j \leq m$, where $V_j=\{v_{j,1},v_{j,2},\cdots,v_{j,k_j}\}$, $v_{j,0}=v_{j,k_j}$, and $p_{j,k_j+1}=p_{j,1}$.
\end{Definition}

In particular, the weights at $p_{j,i}$ are $\{v_{j,i},-v_{j,i-1}\}$ for $1 \leq i \leq k_j$ and $1 \leq j \leq m$. See Section \ref{s2} for basic examples of (families of multi-)fans describing 4-dimensional (almost) complex torus manifolds. Next, to a 4-dimensional almost complex torus manifold, we associate a certain type of graph that contains information about its fixed point data.

\begin{Definition} \label{td14} \cite{J6}
A \textbf{labeled directed k-multigraph} $\Gamma$ is a set $V$ of vertices, a set $E$ of
edges, maps $i : E \to V$ and $t : E \to V$ giving the initial and terminal vertices of
each edge, and a map $w:E \to \mathbb{Z}^k$ giving the label of each edge.
\end{Definition}

A multigraph is a \textbf{graph} if it has no multiple edges between any two vertices.

\begin{Definition}
Let $\Gamma$ be a labeled directed k-graph. For an edge $e$, we say that $(i(e),t(e))$ is $w(e)$-edge, and $(t(e),i(e))$ is $(-w(e))$-edge.
\end{Definition}

Let $w=(w_1,\cdots,w_k) \in \mathbb{Z}^k$. Denote by $\ker w$ the subgroup of $T^k$ whose elements fix $w$;
\begin{center}
$\ker w=\{g=(g_1,\cdots,g_k) \in T^k \subset \mathbb{C}^k \, | \, g^w=g_1^{w_1} \cdots g_k^{w_k}=1\}$.
\end{center}

\begin{Definition} \label{d12} \cite{J6}
Let a $k$-dimensional torus $T^k$ act on a compact almost complex manifold $M$ with isolated fixed points. We say that a (labeled directed k-)multigraph $\Gamma=(V,E)$ \textbf{describes} $M$ if the following hold:
\begin{enumerate}[(i)]
\item The vertex set $V$ is equal to the fixed point set $M^{T^k}$.
\item The multiset of the weights at $p$ is $\{w(e) \, | \, i(e)=p\} \cup \{-w(e) \, | \, t(e)=p\}$ for all $p \in M^{T^k}$.
\item For each edge $e$, the two endpoints $i(e)$ and $t(e)$ are in the same component of the isotropy submanifold $M^{\ker w(e)}$.
\end{enumerate}
\end{Definition}

For an edge $e$, the fixed point $i(e)$ has weight $w(e)$ and the fixed point $t(e)$ has weight $-w(e)$.

Definition \ref{d12} is an extension to a torus action, of a multigraph for an $S^1$-action on a compact almost complex manifold with isolated fixed points in \cite{JT, J5}. For an $S^1$-action, Godinho and Sabatini also considered a multigraph \cite{GS}. For a torus action on a compact almost complex manifold with isolated fixed points, if the weights at each fixed point are pairwise linearly independent (and the odd cohomology groups of the manifold vanish), a multigraph of Definition \ref{d12} is called a signed GKM graph in \cite{GKZ}.

Let $M$ be a 4-dimensional almost complex torus manifold. Let $\Gamma$ be a (labeled directed 2-)graph describing $M$. Definition \ref{d12} implies that $\Gamma$ not only contains the multisets $\{w_{p1},w_{p2}\}$ of the weights at the fixed points $p \in M^{T^2}$, but also contains information that inside $M$ there are chains of isotropy 2-spheres $e$ (which correspond to edges of $\Gamma$), each of which connects two fixed points $i(e)$ and $t(e)$; see Proposition \ref{tp23}. A graph describing $M$ is always admissible, see Proposition \ref{tp24}. In this case, $\Gamma$ naturally gives rise to a family $\Delta$ of admissible multi-fans describing $M$; see Lemma \ref{tl21}. Therefore, a family of multi-fans describing $M$ and a graph describing $M$ each determine a neighborhood of 2-spheres connecting fixed points. Consequently, these become invariants of a 4-dimensional almost complex torus manifold. Any equivariant biholomorphism $f:M_1 \to M_2$ between two 4-dimensional almost complex torus manifolds $M_1$ and $M_2$ takes fixed points to fixed points and takes $w$-spheres to $w$-spheres.

\begin{Definition} \label{td17}
Let $\Gamma$ be a 2-regular labeled directed 2-graph. We say that $\Gamma$ is \textbf{admissible} if the following hold.
\begin{enumerate}[(1)]
\item For any vertex $v$, the labels of the edges of $v$ form a basis of $\mathbb{Z}^2$.
\item If $(p_1,p_2)$ is $w_1$-edge, $(p_2,p_3)$ is $w_2$-edge, and $(p_3,p_4)$ is $w_3$-edge, then $w_3=-aw_2-w_1$ for some integer $a$.
\end{enumerate}
\end{Definition}

For a 4-dimensional almost complex torus manifold $M$, we provide a necessary and sufficient condition for the data regarding the fixed point set. This is accomplished by establishing a necessary and sufficient condition for a family of multi-fans that describes $M$. Equivalently, we establish a necessary and sufficient condition for a graph that describes $M$.

\begin{theorem} \label{tt19}
There exists a 4-dimensional almost complex torus manifold if and only if there exists a finite family of admissible multi-fans describing $M$ if and only if there exists an admissible labeled directed 2-graph describing $M$.
\end{theorem}

In this paper, we do not give an equivariant classification of 4-dimensional almost complex torus manifolds, while we give one for complex manifolds.
Given a 4-dimensional almost complex torus manifold, one may partially recover the equivariant diffemorphism type of the manifold from a family of multi-fans and/or a graph describing it. For this, one may consider the classification of $T^2$-actions on oriented 4-manifolds \cite{OR1, OR2} and see which of them admit an almost complex structure \cite{Ku}.

We inform that for an almost complex torus manifold, Masuda associated a multi-fan in terms of characteristic submanifolds, which are certain codimension 2 submanifolds fixed by an action of some subcircle of the torus \cite{M}; also see \cite{HM}. In contrast, we associate a family of multi-fans, which provides a clearer representation of the geometry of the given manifold.

\subsection{Minimal model and operations}\label{s1.3}

In this subsection, we present a minimal model and operations for each of a family of multi-fans and a graph associated to a 4-dimensional almost complex torus manifold.

\begin{Definition} \label{td110}
We say that a multi-fan $V=\{v_1,v_2,\cdots,v_k\}$ is \textbf{minimal} if each $v_i$ is a unit vector in $\mathbb{Z}^2$.
\end{Definition}

In other words, a multi-fan $V=\{v_1,v_2,\cdots,v_k\}$ is minimal if $v_{4j+1}=(1,0)$, $v_{4j+2}=(0,a)$, $v_{4j+3}=(-1,0)$, $v_{4j+4}=(0,-a)$ for some $a \in \{-1,1\}$, for all $0 \leq j \leq s-1$ where $k=4s$, up to shifting of indices. By definition, a minimal multi-fan has $4s$ vectors for some $s \in \mathbb{N}$. We introduce a corresponding notion for a graph.

\begin{Definition} \label{td111}
We say that a 2-regular connected labeled directed 2-graph is \textbf{minimal} if we can label successive vertices by $p_1$, $p_2$, $\cdots$, $p_{4k}$ so that for $0 \leq j \leq k-1$, $(p_{4j+1},p_{4j+2})$ is $(1,0)$-edge, $(p_{4j+2},p_{4j+3})$ is $(0,a)$-edge, $(p_{4j+3},p_{4j+4})$ is $(-1,0)$-edge, and $(p_{4j+4},p_{4j+5})$ is $(0,-a)$-edge for some $a \in \{-1,1\}$ (with $p_{4k+1}:=p_1$). We say that a 2-regular labeled directed 2-graph $\Gamma$ is \textbf{minimal} if each component of $\Gamma$ is minimal.
\end{Definition}

Next, we define blow up and blow down of a multi-fan and a graph; these correspond to equivariant blow up and blow down of a 4-dimensional almost complex torus manifold.

\begin{Definition}[\textbf{Blow up and blow down of a multi-fan}] \label{td112}
Let $V=\{v_1,v_2,\cdots,v_k\}$ be a multi-fan. Denote $v_{k+1}:=v_1$ and $v_0:=v_k$.
\begin{enumerate}[(1)]
\item A \textbf{blow up} of $(v_i,v_{i+1})$ is an operation that inserts their sum $v_i+v_{i+1}$ between $v_i$ and $v_{i+1}$. That is, a blow up of $(v_i,v_{i+1})$ in $V$ is a multi-fan $V'=\{v_1,\cdots,v_i,v_i+v_{i+1},v_{i+1},\cdots,v_k\}$ (Figure \ref{fig1-2}). 
\item Suppose that three successive vectors $v_i$, $v_{i+1}$, and $v_{i+2}$ satisfy $v_{i+1}=v_i+v_{i+2}$ for some $i$ (Figure \ref{fig2-1}). A \textbf{blow down} of $v_{i+1}$ is an operation that deletes $v_{i+1}$ (Figure \ref{fig2-2}). That is, a blow down of $v_{i+1}$ in $V$ is a multi-fan $V'=\{v_1,\cdots,v_i,v_{i+2},\cdots,v_k\}$.
\end{enumerate}
\end{Definition}

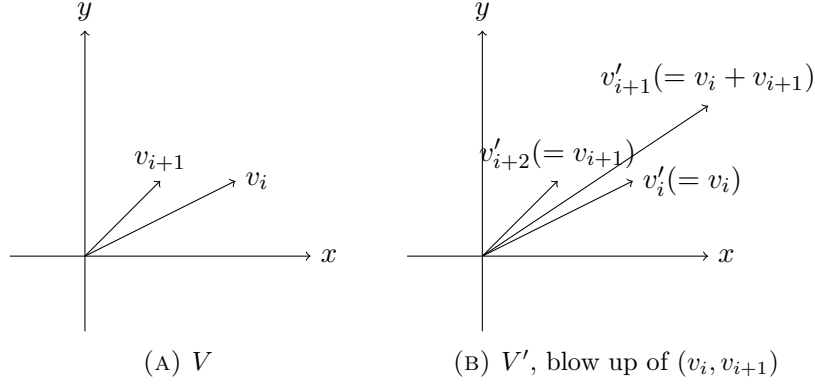
\begin{figure}
\centering
\begin{subfigure}[b][4.8cm][s]{.45\textwidth}
\centering
\begin{tikzpicture}
\draw[->] (-1, 0) -- (3, 0) node[right] {$x$};
\draw[->] (0, -1) -- (0, 3) node[above] {$y$};
\draw[->] (0, 0) -- (2, 1) node[right] {$v_{i}$};
\draw[->] (0, 0) -- (1, 1) node[above] {$v_{i+1}$};
\end{tikzpicture}
\caption{$V$}\label{fig1-1}
\end{subfigure}
\begin{subfigure}[b][4.8cm][s]{.45\textwidth}
\centering
\begin{tikzpicture}
\draw[->] (-1, 0) -- (3, 0) node[right] {$x$};
\draw[->] (0, -1) -- (0, 3) node[above] {$y$};
\draw[->] (0, 0) -- (2, 1) node[right] {$v_i'(=v_{i})$};
\draw[->] (0, 0) -- (1, 1) node[above] {$v_{i+2}'(=v_{i+1})$};
\draw[->] (0, 0) -- (3, 2) node[above] {$v_{i+1}'(=v_{i}+v_{i+1})$};
\end{tikzpicture}
\caption{$V'$, blow up of $(v_{i},v_{i+1})$}\label{fig1-2}
\end{subfigure}\qquad
\caption{Blow up of a multi-fan}\label{fig1}
\end{figure}

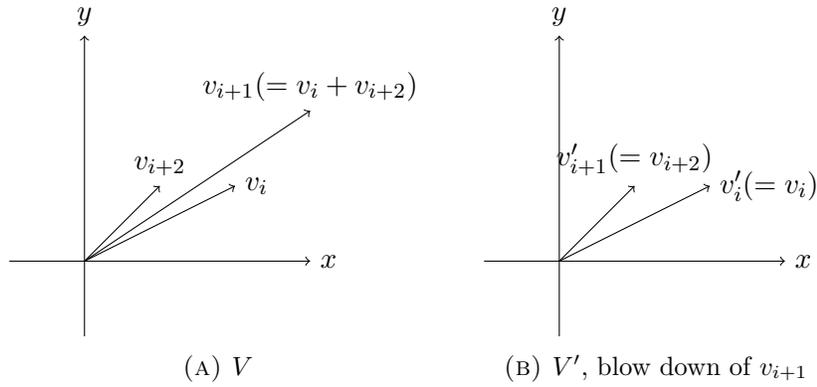
\begin{figure}
\centering
\begin{subfigure}[b][5.4cm][s]{.45\textwidth}
\centering
\begin{tikzpicture}
\draw[->] (-1, 0) -- (3, 0) node[right] {$x$};
\draw[->] (0, -1) -- (0, 3) node[above] {$y$};
\draw[->] (0, 0) -- (2, 1) node[right] {$v_{i}$};
\draw[->] (0, 0) -- (1, 1) node[above] {$v_{i+2}$};
\draw[->] (0, 0) -- (3, 2) node[above] {$v_{i+1}(=v_i+v_{i+2})$};
\end{tikzpicture}
\caption{$V$}\label{fig2-1}
\end{subfigure}
\begin{subfigure}[b][5.4cm][s]{.45\textwidth}
\centering
\begin{tikzpicture}
\draw[->] (-1, 0) -- (3, 0) node[right] {$x$};
\draw[->] (0, -1) -- (0, 3) node[above] {$y$};
\draw[->] (0, 0) -- (2, 1) node[right] {$v_{i}'(=v_i)$};
\draw[->] (0, 0) -- (1, 1) node[above] {$v_{i+1}'(=v_{i+2})$};
\end{tikzpicture}
\caption{$V'$, blow down of $v_{i+1}$}\label{fig2-2}
\end{subfigure}\qquad
\caption{Blow down of a multi-fan}\label{fig2}
\end{figure}

If $v_i$ and $v_{i+1}$ are in (counter)clockwise order, then $v_i$, $v_i+v_{i+1}$, $v_{i+1}$ are also in (counter)clockwise order. Therefore, blow up of a multi-fan is also a multi-fan. Similarly, blow down of a multi-fan is also a multi-fan.

\begin{Definition}[\textbf{Blow up and down of a graph}] \label{td113}
Let $\Gamma$ be a 2-regular labeled directed graph.
\begin{enumerate}[(1)]
\item Suppose that $(p',p)$ is $w_1$-edge and $(p,p'')$ is $w_2$-edge (Figure \ref{tfig1-1}). A \textbf{blow up} of $p$ is an operation that replaces $p$ with $(w_1+w_2)$-edge $(p_1,p_2)$ so that $(p',p_1)$ is $w_1$-edge and $(p_2,p'')$ is $w_2$-edge (Figure \ref{tfig1-2}).
\item Suppose that $(p',p_1)$ is $w_1$-edge, $(p_1,p_2)$ is $(w_1+w_2)$-edge, and $(p_2,p'')$ is $w_2$-edge (Figure \ref{tfig1-2}). A \textbf{blow down} of $(p_1,p_2)$ is an operation that shrinks the edge $(p_1,p_2)$ to a vertex $p$ (Figure \ref{tfig1-1}).
\end{enumerate}
\end{Definition}

\begin{figure}
\centering
\begin{subfigure}[b][6.2cm][s]{.4\textwidth}
\centering
\vfill
\begin{tikzpicture}[state/.style ={circle, draw}]
\node[state] (a) {$p'$};
\node[state] (b) [above right=of a] {$p$};
\node[state] (c) [above left=of b] {$p''$};
\path (a) [->] edge node[right] {$w_1$} (b);
\path (b) [->] edge node [right] {$w_2$} (c);
\end{tikzpicture}
\vfill
\caption{Blow down of $(p_1, p_2)$}\label{tfig1-1}
\end{subfigure}
\begin{subfigure}[b][6.2cm][s]{.4\textwidth}
\centering
\begin{tikzpicture}[state/.style ={circle, draw}]
\node[state] (a) {$p'$};
\node[state] (b) [above right=of a] {$p_1$};
\node[state] (c) [above=of b] {$p_2$};
\node[state] (d) [above left=of c] {$p''$};
\path (a) [->] edge node[right] {$w_1$} (b);
\path (b) [->] edge node [left] {$w_1+w_2$} (c);
\path (c) [->] edge node [left] {$w_2$} (d);
\end{tikzpicture}
\caption{Blow up of $p$}\label{tfig1-2}
\end{subfigure}\qquad
\caption{Blow up and blow down of a graph}\label{tfig1}
\end{figure}
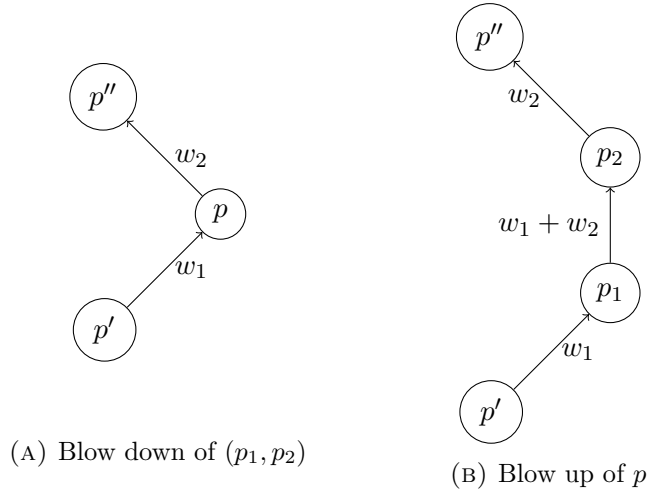

In the case that the almost complex structure is locally integrable, blow up and down of a 4-dimensional almost complex torus manifold, a family of multi-fans describing it, and a graph describing it all correspond to each other (Proposition \ref{tp43}). With these notions, we state the main result of this subsection.

\begin{theorem} \label{tt114}
Let $M$ be a 4-dimensional almost complex torus manifold. Let $\Delta$ and $\Gamma$ be a family of multi-fans and a graph describing $M$. Then the following hold.
\begin{enumerate}[(1)]
\item We can blow up and down $\Delta$ to a family $\Delta'$ of minimal multi-fans,
\item We can blow up and down $\Gamma$ to a minimal graph $\Gamma'$.
\end{enumerate}
Assume furthermore that the almost complex structure is integrable near each $w$-sphere. Then we can equivariantly blow up and down $M$ to another 4-dimensional almost complex torus manifold $M'$ whose weights at the fixed points are all unit vectors in $\mathbb{Z}^2$; moreover, we can take blow up and down of a manifold, a family of multi-fans, and a graph to correspond to each other, so that the family $\Delta'$ of minimal multi-fans and the minimal graph $\Gamma'$ each describe $M'$. 
\end{theorem}

Since blow up and blow down are reverse operations of each other, Theorem \ref{tt114} implies that if $\Delta$ ($\Gamma$) is a family of multi-fans (a graph) describing a 4-dimensional almost complex torus manifold, then $\Delta$ ($\Gamma$) is obtained from some family of minimal multi-fans (some minimal graph) by blow up and down.

By Theorem \ref{tt114}, for an equivariant classification of 4-dimensional almost complex torus manifolds, what remains is to classify minimal manifolds, those that are described by families of minimal multi-fans, alternatively, by minimal graphs. 

Moreover, for any family of multi-fans (any graph) obtained from a family of minimal multi-fans (a minimal graph), there exists a manifold described by each of them.

\begin{theorem} \label{tt114a}
Let $\Delta$ ($\Gamma$) be a family of multi-fans (a graph) obtained from a finite family of minimal multi-fans (a minimal 2-graph) by blow up and blow down. Then there exists a 4-dimensional almost complex torus manifold described by $\Delta$ ($\Gamma$).
\end{theorem}

We discuss in the next subsection that if $M$ is complex, then a family of multi-fans describing $M$ is a fan, and both the fan and a graph describing $M$ each determine the manifold.

\subsection{Applications to complex manifolds} \label{s1.4}

In this subsection, we discuss applications of our results to complex manifolds. In algebraic geometry, a minimal surface is one that cannot be blown down further. When a given 4-dimensional almost complex torus manifold is a complex manifold, we can blow up and down it to a minimal manifold in our sense, which is $\mathbb{CP}^1 \times \mathbb{CP}^1$.

\begin{theo} \label{tt115}
Let $M$ be a 4-dimensional complex torus manifold. Then we can equivariantly blow up and down $M$ to $\mathbb{CP}^1 \times \mathbb{CP}^1$ with a $T^2$-action
\begin{center}
$(t_1,t_2) \cdot ([z_0:z_1],[y_0:y_1])=([z_0: t_1 z_1],[y_0:t_2 y_1])$ 
\end{center}
for all $(t_1,t_2) \in T^2 \subset \mathbb{C}^2$ and $([z_0:z_1].[y_0:y_1]) \in \mathbb{CP}^1 \times \mathbb{CP}^1$. \end{theo}

Theorem \ref{tt115} implies the following corollary.

\begin{cor} \label{tc116}
Any two 4-dimensional complex torus manifolds can be obtained from each other by equivariant blow up and blow down. 
\end{cor}

Analogous to a correspondence between toric varieties and fans, we establish one between 4-dimensional complex torus manifolds and admissible fans.

\begin{cor} \label{tc117}
An admissible fan determines a 4-dimensional complex torus manifold up to equivariant biholomorphism. In other words, two 4-dimensional complex torus manifolds are equivariantly biholomorphic if and only if their fans agree.
\end{cor}

Since any admissible fan can be obtained from a minimal fan by blow up and blow down, we can rephrase Corollary \ref{tc117} as follows.

\begin{cor} \label{tc118}
Four dimensional complex torus manifolds are classified by blow up and blow down of the minimal fan $\{(1,0), (0,1), (-1,0), (0,-1)\}$.
\end{cor}

Note that while we obtained these results as applications of the results for almost complex torus manifolds to complex torus manifolds, one may directly prove by combining the work of \cite{IK} and \cite{T}.

A (non-singular) rational surface is a complex surface that is obtained from a minimal rational surface by blow ups. The minimal rational surfaces are the complex projective space $\mathbb{CP}^2$ and the Hirzebruch surfaces. For a complex manifold, our blow up and blow down for an almost complex manifold coincide with the blow up and blow down in complex geometry. Thus, Corollary \ref{tc116} also implies the following well-known fact.

\begin{cor} \label{tc119}
Any two rational surfaces can be obtained from each other by blow up and blow down.
\end{cor}

Let $M$ be a 4-dimensional complex torus manifold. According to Theorem \ref{tt19}, there exists an admissible graph $\Gamma=(V,E)$ describing $M$. By Lemma \ref{connected}, $\Gamma$ is connected. Theorem \ref{tt115} implies that the odd degree cohomology groups of $M$ vanish, and thus the $T^2$-action on $M$ is equivariantly formal. Since the weights at each fixed point are linearly independent (these form a basis of $\mathbb{Z}^2$), the graph $\Gamma$ is a GKM graph if we forget the direction of each edge \cite[Corollary 2.4]{J6}. Then, by the GKM theory \cite{GKM}, the graph $\Gamma$ encodes the equivariant cohomology $H_{T^2}^*(M;\mathbb{Q})$ of $M$ as follows.
\begin{center}
$\displaystyle H_{T^2}^*(M;\mathbb{Q}) \simeq \{(f_{i(e)}) \in \bigoplus_{i(e) \in M^{T^2}} H^*(BT^2;\mathbb{Q}) : f_{i(e)}-f_{t(e)} \in (w(e)), \forall e \in E\}$.
\end{center}
Here, for each edge $e \in E$, $(w(e))$ is the ideal generated by the weight $w(e) \in H^2(BT^2;\mathbb{Q})$ and $BT^2$ is the classifying space of $T^2$.

\begin{cor} \label{c120}
Let $M$ be a 4-dimensional complex torus manifold. Then there exists a connected admissible graph describing $M$, which encodes the equivariant cohomology of $M$.
\end{cor}

\subsection{Organization} \label{s1.5} We will now outline the structure of this paper.
\begin{enumerate}[(1)]
\item In section \ref{s2}, we present basic examples of 4-dimensional (almost) complex torus manifolds, along with their corresponding (families of multi-)fans and graphs. Specifically, we discuss the complex projective space $\mathbb{CP}^2$ and the Hirzebruch surfaces. 
\item Section \ref{s3} introduces the construction of a 4-dimensional almost complex torus manifold using equivariant plumbing, starting from a family of multi-fans.
\item In Section \ref{s5.1}, we establish a correspondence between families of admissible multi-fans and admissible graphs (Lemma \ref{tl21}). We also prove the existence of a family of multi-fans and a graph describing any 4-dimensional almost complex torus manifold (Proposition \ref{tp24}). Finally, we provide a proof of Theorem \ref{tt19}.
\item Section \ref{s5.2} explores the topic of equivariant blow up and blow down of a 4-dimensional almost complex torus manifold. It demonstrates that blow up and blow down operations for the manifold, its corresponding family of multi-fans, and its graph are all interrelated (Proposition \ref{tp43}).
\item In Section \ref{s5.3}, we prove Theorem \ref{tt114} and Theorem \ref{tt114a}.
\item Section \ref{s5.4} delves into the study of 4-dimensional complex torus manifolds. We present the proof of Theorem \ref{tt115} and Corollary \ref{tc117}.
\item In Section \ref{s5.5}, we investigate various properties of 4-dimensional almost complex torus manifolds. This includes establishing a lower bound on the number of fixed points for such manifolds (Proposition \ref{tp75}). Additionally, we explore the relationship between the Hirzebruch $\chi_y$-genus and blow up and down operations (Lemma \ref{tl77}). Finally, we provide the necessary and sufficient conditions for both the Hirzebruch $\chi_y$-genus (Theorem \ref{tt78}) and the Chern numbers (Theorem \ref{tt79}).
\item In Section \ref{s5.6}, we classify 4-dimensional almost complex torus manifolds with few fixed points.
\end{enumerate}

\section{Examples} \label{s2}

In this section, we present basic examples of 4-dimensional (almost) complex torus manifolds, the complex projective space $\mathbb{CP}^2$ and the Hirzebruch surfaces, along with their fans and graphs.

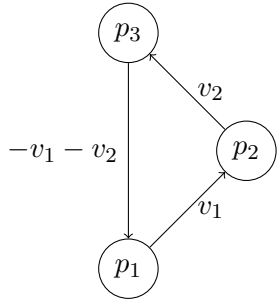
\begin{figure}
\begin{tikzpicture}[state/.style ={circle, draw}]
\node[state] (a) {$p_1$};
\node[state] (b) [above right=of a] {$p_2$};
\node[state] (c) [above left=of b] {$p_3$};
\path (a) [->] edge node [right] {$v_1$} (b);
\path (b) [->] edge node [right] {$v_2$} (c);
\path (c) [->] edge node [left] {$-v_1-v_2$} (a);
\end{tikzpicture}
\caption{Graph describing $\mathbb{CP}^2$}\label{tfig14}
\end{figure}

\begin{exa} \label{te3}
Let $v_1$ and $v_2$ be a basis of $\mathbb{Z}^2$. Let a torus $T^2$ act on $\mathbb{CP}^2$ by
\begin{center}
$g \cdot [z_0:z_1:z_2]=[z_0:g^{v_1}z_1:g^{v_1+v_2}z_2]$
\end{center}
for all $g \in T^2 \subset \mathbb{C}^2$ and for all $[z_0:z_1:z_2] \in \mathbb{CP}^2$. There are 3 fixed points $p_1=[1:0:0]$, $p_2=[0:1:0]$, and $p_3=[0:0:1]$, and the weights at these fixed points are $\{v_1,v_1+v_2\}$, $\{-v_1,v_2\}$, and $\{-v_2,-v_1-v_2\}$, respectively. 

The fixed points $p_1$ and $p_2$ lie in the 2-sphere $\{[z_0:z_1:0]\}$ upon which $T^2$ acts by $g \cdot [z_0:z_1:0]=[z_0:g^{v_1}:0]$ giving $p_1$ weight $v_1$ and $p_2$ weight $-v_1$; thus $(p_1,p_2)$ is $v_1$-sphere. Similarly, $p_1$ and $p_3$ are in the 2-sphere $\{[z_0:0:z_2]\}$ upon which $T^2$ acts by $g \cdot [z_0:0:z_2]=[z_0:0:g^{v_1+v_2}]$ and hence $(p_1,p_3)$ is $(v_1+v_2)$-sphere; in other words, $(p_3,p_1)$ is $-(v_1+v_2)$-sphere. Also, $p_2$ and $p_3$ are in the 2-sphere $\{[0:z_1:z_2]\}$ upon which $T^2$ acts by $g \cdot [0:z_1:z_2]=[0:g^{v_1}z_1:g^{v_1+v_2}z_2]$; thus $(p_2,p_3)$ is $v_2$-sphere.

Therefore, the fan $V=\{v_1,v_2,-v_1-v_2\}$ and the graph Figure \ref{tfig14} each describe this action on $\mathbb{CP}^2$.
\end{exa}

\begin{figure}
\begin{tikzpicture}[state/.style ={circle, draw}]
\node[state] (A) {$p_1$};
\node[state] (B) [above left=of A] {$p_4$};
\node[state] (C) [above right=of A] {$p_2$};
\node[state] (D) [above right=of B] {$p_3$};
\path (B) [->] edge node[left] {$-v_2$} (A);
\path (A) [->] edge node [right] {$v_1$} (C);
\path (D) [->] edge node [left] {$-v_1+nv_2$} (B);
\path (C) [->] edge node [right] {$v_2$} (D);
\end{tikzpicture}
\caption{Graph describing Hirzebruch surface $\Sigma_n$}\label{tfig15}
\end{figure}
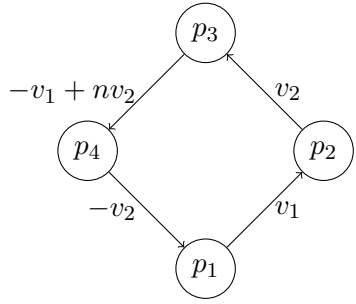

\begin{exa} \label{te4}
Let $n$ be an integer. The Hirzebruch surface $\Sigma_n$ is a complex surface defined by
\begin{center}
$\Sigma_n=\{([z_0:z_1:z_2],[y_1:y_2])\in\mathbb{CP}^2\times\mathbb{CP}^1:z_1 y_2^n=z_2 y_1^n\}$.
\end{center}
Let $v_1$ and $v_2$ form a basis of $\mathbb{Z}^2$. Let a torus $T^2$ act on the Hirzebruch surface  by
\begin{center}
$g \cdot ([z_0:z_1:z_2],[y_1:y_2]) = ([g^{v_1} z_0:z_1:g^{n v_2} z_2],[y_1:g^{v_2} y_2])$
\end{center}
for all $g \in T^2 \subset \mathbb{C}^2$ and for all $([z_0:z_1:z_2],[y_1:y_2]) \in \Sigma_n$. The action has 4 fixed points, $p_1=([0:1:0],[1:0])$, $p_2=([1:0:0],[1:0])$, $p_3=([1:0:0],[0:1])$, $p_4=([0:0:1],[0:1])$ that have weights $\{v_1,v_2\}$, $\{-v_1,v_2\}$, $\{-v_1+nv_2,-v_2\}$, $\{v_1-nv_2,-v_2\}$, respectively. 

The fixed points $p_1$ and $p_2$ lie in the 2-sphere $\{([z_0:z_1:0],[1:0])\}$ upon which $T^2$ acts by $g \cdot ([z_0:z_1:0],[1:0])=([g^{v_1} z_0:z_1:0],[1:0])$; hence $(p_1,p_2)$ is $v_1$-sphere. Similarly, $p_2$ and $p_3$ lie in the 2-sphere $\{([1:0:0],[y_1:y_2])\}$, upon which $T^2$ acts by $g \cdot ([1:0:0],[y_1:y_2])=([1:0:0],[w_1:g^{v_2} w_2])$ and hence $(p_2,p_3)$ is $v_2$-sphere. Also, $p_3$ and $p_4$ lie in the 2-sphere $\{([z_0:0:z_2],[0:1])\}$ upon which $T^2$ acts by $g \cdot ([z_0:0:z_2],[0:1])=([g^{v_1} z_0:0:g^{nv_2} z_2],[0:1])$ and hence $(p_3,p_4)$ is $(-v_1+nv_2)$-sphere. Finally, $p_1$ and $p_4$ lie in the 2-sphere $\{([0:z_1:z_2],[y_1:y_2])\}$ upon which $T^2$ acts by $g \cdot ([0:z_1:z_2],[y_1:y_2])=([0:z_1:g^{nv_2}z_2],[y_1:g^{v_2} y_2])$ and hence $(p_1,p_4)$ is $v_2$-sphere; in other words, $(p_4,p_1)$ is $(-v_2)$-sphere.

Therefore, the fan $\{v_1,v_2,-v_1+nv_2,-v_2\}$ and the graph Figure \ref{tfig15} each describe this action on $\Sigma_n$.
\end{exa}

\section{Constructing 4-dimensional almost complex torus manifolds} \label{s3}

In this section, we use the equivariant plumbing to construct a 4-dimensional almost complex torus manifold from a given multi-fan.

Let $a$ be an integer. Let $\mathcal{O}(a)$ denote the holomorphic line bundle over $\mathbb{CP}^1$ whose self-intersection number of the zero section is $a$. That is, the quotient of $(\mathbb{C}^2-\{0\}) \times \mathbb{C}$ by a $C^*$-action given by 
\begin{center}
$z \cdot (z_1,z_2,w)=(zz_1,zz_2,z^aw)$,
\end{center}
for all $z \in C^*$ and $(z_1,z_2,w) \in (\mathbb{C}^2-\{0\}) \times \mathbb{C}$. For $(z_1,z_2,w) \in (\mathbb{C}^2-\{0\}) \times \mathbb{C}$, we denote by $[z_1,z_2,w]$ its equivalence class.

Let $T^2$ act on $\mathcal{O}(a)$ by 
\begin{center}
$g \cdot [z_1,z_2,w]=[z_1,g^{u_1}z_2,g^{u_2}w]$
\end{center}
for all $g \in T^2$, for some non-zero $u_1,u_2 \in \mathbb{Z}^2$ such that $au_1 \neq u_2$. The action has two fixed points $q_1=[1,0,0]$ and $q_2=[0,1,0]$, and the weights at these fixed points $q_1$ and $q_2$ are $\{u_1,u_2\}$ and $\{-u_1,-au_1+u_2\}$, respectively. 

In the below theorem, given an admissible multi-fan, we use the equivariant plumbing to construct a 4-dimensional almost complex torus manifold described by the multi-fan. This construction is a slight modification of the construction in \cite{J5}, which closely follows the idea of a construction of a 4-dimensional almost complex torus manifold in Theorem 5.1 of \cite{M} but in a different way; the difference is explained in \cite{J5}. We restate and modify the statement and proof to suit our purposes.

\begin{theo} \label{tt31} Let $V=\{v_1,v_2,\cdots,v_k\}$ be an admissible multi-fan. Then there exists a 4-dimensional almost complex torus manifold described by $V$.
\end{theo}

\begin{proof}
Denote $v_0:=v_k$ and $v_{k+1}:=v_1$. Since $V$ is admissible, for each $1 \leq i \leq k$,
\begin{enumerate}
\item $v_{i-1}$ and $v_i$ form a basis of $\mathbb{Z}^2$, and
\item $v_{i+1}=-a_i v_i-v_{i-1}$ for some integer $a_i$.
\end{enumerate}
For each $1 \leq i \leq k$, let $D_{a_i}(v_i,-v_{i-1})$ be the disk bundle of the holomorphic line bundle $\mathcal{O}(a_i)$ over $\mathbb{CP}^1$ whose self-intersection number of the zero section is $a_i$, where $T^2$ acts on $\mathcal{O}(a_i)$ by
\begin{center}
$g \cdot [z_1,z_2,w]=[z_1, g^{v_i}z_2, g^{-v_{i-1}}w]$
\end{center}
for all $g \in T^2$. This action has two fixed points $q_{i,1}=[1,0,0]$ and $q_{i,2}=[0,1,0]$ that have weights $\{v_{i},-v_{i-1}\}$ and $\{-a_i v_i-v_{i-1},-v_i\}$, respectively. Let $q_{k+1,j}:=q_{1,j}$ for $j \in \{1,2\}$. Since $v_{i+1}=-a_i v_i-v_{i-1}$, the weights $\{-a_i v_i-v_{i-1},-v_i\}$ at $q_{i,2}$ and the weights $\{v_{i+1},-v_i\}$ at $q_{i+1,1}$ agree for each $i \in \{1,\cdots,k\}$. Hence, for each $1 \leq i \leq k$, we can equivariantly plumb two manifolds $D_{a_i}(v_i,-v_{i-1})$ and $D_{a_{i+1}}(v_{i+1},-v_i)$ at $q_{i,2}$ and $q_{i+1,1}$ (Here, $a_0:=a_k$ and $a_{k+1}:=a_1$). As a result, we get a 4-dimensional compact connected manifold $N$ with a $T^2$-action. As in the proof of Theorem 5.1 of \cite{M}, by pasting $N$ and $D^2 \times T^2$ along the boundary $\partial N \simeq S^1 \times T^2$ of $N$, we get a 4-dimensional almost complex torus manifold $M$; it admits a $T^2$-invariant almost complex structure.

By the construction, the $T^2$-action on $M$ has $k$ fixed points $p_i:=q_{i,1}$, $1 \leq i \leq k$, and $(p_i,p_{i+1})$ is $v_i$-sphere for each $1 \leq i \leq k$ ($p_{k+1}:=p_1$). In particular, the weights at $p_i$ are $\{v_i,-v_{i-1}\}$ for $1 \leq i \leq k$. Thus, $V$ describes $M$. \end{proof}

\section{Correspondence between graphs and families of multi-fans} \label{s5.1}

In this section, we establish a one-to-one correspondence between families of admissible multi-fans and admissible 2-graphs (Lemma \ref{tl21}). We also demonstrate that for a 4-dimensional almost complex torus manifold, there exist a family of multi-fans and a graph describing it (Proposition \ref{tp24}). Moreover, we show that a family of admissible multi-fans accurately describes some 4-dimensional almost complex torus manifold (Theorem \ref{tt33}). These results lead to the proof of Theorem \ref{tt19}, which is provided at the end of this section.

For an action of a torus $T^k$ on a compact almost complex manifold $M$ with isolated fixed points, there exists a multigraph describing $M$ that has no self-loops \cite{J6}. If $M$ is an almost complex torus manifold, that is, $k=\frac{1}{2} \dim M$, such a multigraph is a graph.

\begin{pro} \label{tp22} \cite{J6}
Let $M$ be an almost complex torus manifold. Then there exists a graph describing $M$ that has no self-loops.
\end{pro}

In a graph describing an almost complex torus manifold, each edge corresponds to an isotropy sphere (see Definition~\ref{tsphere}.)

\begin{pro} \label{tp23}
Let $M$ be an almost complex torus manifold. Let $\Gamma$ be a graph describing $M$. If $(p,q)$ is $w$-edge in $\Gamma$, then $(p,q)$ is $w$-sphere.
\end{pro}

\begin{proof}
Suppose $(p,q)$ is $w$-edge. By Definition \ref{d12}, two fixed points $p$ and $q$ are in the same component $F$ of $M^{\ker w}$, which is a compact almost complex submanifold. Let $\dim M=2n$. Because the weights at $p$ form a basis of $\mathbb{Z}^n$, no other weights at $p$ are multiples of $w$. This implies that $F$ is 2-dimensional. The $T^n$-action on $M$ restricts to act on $F$, having $p$ and $q$ as fixed points with weights $w$ and $-w$ for this action, respectively. That is, $(p,q)$ is $w$-sphere.
\end{proof}

The proof of Proposition \ref{tp23} explains how to associate a graph to an almost complex torus manifold. When a fixed point $p$ has weight $w$, we consider a component $F$ of $M^{\ker w}$ that contains $p$; then $F$ is 2-dimensional, and thus is $\mathbb{CP}^1$, and so has another fixed point $q$ with weight $-w$. Then we draw an edge from $p$ to $q$ with a label $w$. That a resulting multigraph $\Gamma$ is in fact a graph requires an additional argument \cite[Proposition 2.8]{J6}, but if $\dim M=4$ this can be seen as follows. Suppose that $\dim M=4$ and there are two edges $e_1$ and $e_2$ between $p$ and $q$. Assume each $e_i$ has initial vertex $p$ and terminal vertex $q$; the other cases will be similar. Then $p$ has weights $\{w(e_1),w(e_2)\}$, $q$ has weights $\{-w(e_1),-w(e_2)\}$, and by Proposition \ref{tp23}, $(p,q)$ is $w(e_1)$-sphere. By Lemma \ref{tl23} below, we must have $w(e_2)=-w(e_2)+aw(e_1)$ for some integer $a$, that is, $2w(e_2)=aw(e_1)$. However, $w(e_1)$ and $w(e_2)$ are the weights at $p$, and so they form a basis of $\mathbb{Z}^2$, which leads to a contradiction.

For a 4-dimensional almost complex torus manifold, there exist an admissible graph and a family of admissible multi-fans describing it. First, we establish a one-to-one correspondence between admissible 2-graphs and families of admissible multi-fans.

\begin{lem} \label{tl21}
There is a one-to-one correspondence between admissible 2-graphs and families of admissible multi-fans.
\end{lem}

\begin{proof}
Let $\Gamma$ be an admissible 2-graph. Let $\Gamma_1$, $\cdots$, $\Gamma_m$ denote the connected components of $\Gamma$. Let $p_{j,1}$, $\cdots$, $p_{j,k_j}$ be successive vertices of $\Gamma_j$ for $1 \leq j \leq m$. Let $(p_{j,i},p_{j,i+1})$ be $w_{j,i}$-edge for $1 \leq i \leq k_j$; here, $p_{j,k_j+1}:=p_{j,1}$. Since $\Gamma$ is admissible, $w_{j,i+1}=-a_{j,i} w_{j,i}-w_{j,i-1}$ for some integer $a_{j,i}$, for all $1 \leq i \leq k_j$; here, $w_{j,0}:=w_{j,k_j}$ and $w_{j,k_j+1}:=w_{j,1}$. For $1 \leq j \leq m$, define a multi-fan $V_j$ by $V_j=\{w_{j,1},\cdots, w_{j,k_j}\}$ and let $\Delta$ be the family of $V_j$'s. Since $\Gamma$ is admissible, the elements $\{w_{j,i},-w_{j,i-1}\}$ at $p_{j,i}$ form a basis of $\mathbb{Z}^2$ for all $1 \leq i \leq k_j$ and for all $1 \leq j \leq m$. Moreover, for each $j$, since $w_{j,i+1}=-a_{j,i} w_{j,i}-w_{j,i-1}$ for all $1 \leq i \leq k_j$, $w_{j,i}$'s are in (counter)clockwise order, and hence $V_j$ is admissible. Thus, $\Delta$ is a family of admissible multi-fans.

Conversely, suppose that we are given a family $\Delta$ of admissible multi-fans $V_1$, $\cdots$, $V_m$. For each $1 \leq j \leq m$, let $V_j=\{v_{j,1},\cdots,v_{j,k_j}\}$. Then for each $j$, construct a graph $\Gamma_j$ so that $(p_{j,i},p_{j,i+1})$ is $v_{j,i}$-edge for $1 \leq i \leq k_j$, with the setting $p_{j,k_j+1}:=p_{j,1}$. Since $V_j$ is admissible, $v_{j,i}$ and $v_{j,i+1}$ form a basis of $\mathbb{Z}^2$ and $v_{j,i+1}=-a_{j,i} v_{j,i}-v_{j,i-1}$ for some integer $a_{j,i}$, for each $i$ (here $v_{j,k_j+1}:=v_{j,1}$). Therefore, $\Gamma_j$ is admissible, and hence so is $\Gamma$.
\end{proof}

For a 4-dimensional almost complex $T^k$-manifold, when $(p,q)$ is $w$-sphere, the next lemma shows a relation between the self-intersection number of the sphere and the weights at $p$ and $q$.

\begin{lem} \label{tl23} Let a $k$-dimensional torus $T^k$ act on a 4-dimensional almost complex manifold $M$. Let $(p,q)$ be $w$-sphere, for some $p,q \in M^{T^k}$. Let $p$ and $q$ have weights $\{w,w'\}$ and $\{-w,w''\}$, respectively. Then $w'=w''+aw$, where $a$ is the self-intersection number of the zero section of the normal bundle of the $w$-sphere $(p,q)$. 
\end{lem}

\begin{proof}
Let $F$ denote the $w$-sphere $(p,q)$. By definition, $F$ is a compact almost complex submanifold of $M$. Any almost complex structure on $S^2$ is integrable, and a complex structure on $S^2$ is unique, up to biholomorphism. The normal bundle $NF$ of $F$ is a $T^k$-invariant complex line bundle over $F$, which is isomorphic to some $\mathcal{O}(b)$ with a $T^k$-action
\begin{center}
$g \cdot [z_1,z_2,w]=[z_1,g^{u_1}z_2,g^{u_2}w]$
\end{center}
that has 2 fixed points $[1,0,0]$ and $[0,1,0]$ having weights $\{u_1,u_2\}$ and $\{-u_1,-bu_1+u_2\}$, respectively. Since $b=a$ and $u_1=w$, it follows that $u_2=w'$ and $w''=-bu_1+u_2=-aw+w'$. \end{proof}

For a 4-dimensional almost complex torus manifold, a graph describing it is admissible; thus a family of multi-fans describing it is also admissible.

\begin{pro} \label{tp24}
Let $M$ be a 4-dimensional almost complex torus manifold. Then there exist an admissible 2-graph and a family of admissible multi-fans describing $M$.
\end{pro}

\begin{proof}
By Proposition \ref{tp22}, there is a graph $\Gamma$ describing $M$ that has no self-loops. Suppose that $(p_1,p_2)$ is $w_1$-edge, $(p_2,p_3)$ is $w_2$-edge, and $(p_3,p_4)$ is $w_3$-edge; because $\Gamma$ has no multiple edges, $p_1$ and $p_3$ are necessarily distinct (but it may happen that $p_1=p_4$). By Proposition \ref{tp23}, $(p_2,p_3)$ is $w_2$-sphere. Since $(p_2,p_3)$ is $w_2$-sphere, $p_2$ has weights $\{-w_1,w_2\}$, and $p_3$ has weights $\{-w_2,w_3\}$, by applying Lemma \ref{tl23} to this $w_2$-sphere $(p_2,p_3)$, we conclude that $w_3=-aw_2-w_1$ for some integer $a$. Repeating this argument for any three successive edges, it follows that $\Gamma$ is admissible.

By Lemma \ref{tl21}, $\Gamma$ naturally gives rise to a family $\Delta$ of admissible multi-fans. By Proposition \ref{tp23}, if $(p,q)$ is $w$-edge in $\Gamma$, then $(p,q)$ is $w$-sphere. This implies that $\Delta$ describes $M$. \end{proof}

We have established the following relationships.
\begin{center}
$\{\textrm{4-dimensional almost complex torus manifolds}\} \to \{\textrm{admissible graphs}\} \iff \{\textrm{families of admissible multi-fans}\}$
\end{center}

Given two 4-dimensional almost complex torus manifolds, we can take an equivariant connected sum along free orbits of them to construct another almost complex torus manifold.

\begin{lem} \cite{J5} \label{tl32} For $i=1,2$, let $M_i$ be a 4-dimensional almost complex torus manifold. Then we can perform an equivariant connected sum along free orbits of $M_1$ and $M_2$ to construct a 4-dimensional almost complex torus manifold. \end{lem}

We show that any admissible graph, and hence any family of admissible multi-fans, describes some 4-dimensional almost complex torus manifold.
Combining Theorem \ref{tt31} and Lemma \ref{tl32}, any finite family of admissible multi-fans describes some 4-dimensional almost complex torus manifold.

\begin{theo} \label{tt33}
\begin{enumerate} 
\item Let $\Delta$ be a family of admissible multi-fans $V_1$, $\cdots$, $V_m$. Then there exists a 4-dimensional almost complex torus manifold described by $\Delta$.
\item Let $\Gamma$ be an admissible graph. Then there exists a 4-dimensional almost complex torus manifold described by $\Gamma$.
\end{enumerate}
\end{theo}

\begin{proof}
Since $V_j$ is admissible, by Theorem \ref{tt31}, there exists a 4-dimensional almost complex torus manifold $M_j$ described by $V_j$, $1 \leq j \leq m$. By Lemma \ref{tl32}, we can take an equivariant connected sum along free orbits of $M_j$ to construct a 4-dimensional almost complex torus manifold $M$; then $\Delta$ describes $M$.

The second claim follows from the first claim and Lemma~\ref{tl21}.
\end{proof}

In particular, a family of minimal multi-fans describes some almost complex torus manifold.

\begin{pro} \label{tp34}
\begin{enumerate}
\item Let $\Delta$ be a family of minimal multi-fans $V_1$, $\cdots$, $V_m$. Then there exists a 4-dimensional almost complex torus manifold described by $\Delta$.
\item Let $\Gamma$ be a minimal graph. Then there exists a 4-dimensional almost complex torus manifold described by $\Gamma$.
\end{enumerate}
\end{pro}

\begin{proof}
Let $V_j=\{v_{j,1},\cdots,v_{j,k_j}\}$, $1 \leq j \leq m$. For each $j$, since $V_j$ is a multi-fan, $v_{j,1}$, $\cdots$, $v_{j,k_j}$, $v_{j,1}$ are in counterclockwise order or in clockwise order; moreover, since $V_j$ is minimal, each $v_{j,i}$ is a unit vector in $\mathbb{Z}^2$ (Definition \ref{td110}). These imply that $v_{j,i+1}=-v_{j,i-1}$ for $1 \leq i \leq k_j$ (Here, $v_{j,0}:=v_{j,k}$ and $v_{j,k_j+1}:=v_{j,1}$). Hence each $V_j$ is admissible. Therefore, (1) follows from Theorem \ref{tt33}. 

We prove (2). Let $\Gamma$ be a minimal graph. Then $\Gamma$ is admissible. By Lemma \ref{tl21}, $\Gamma$ gives rise to a family $\Delta$ of admissible multi-fans, which are minimal. By (1), $\Delta$ describes some 4-dimensional almost complex torus manifold $M$; thus, $\Gamma$ also describes $M$. \end{proof}

We have found a necessary and sufficient condition for a family of multi-fans and a graph describing a 4-dimensional almost complex torus manifold, which proves Theorem \ref{tt19}.

\begin{proof}[\textbf{Proof of Theorem \ref{tt19}}]
By Proposition \ref{tp24}, there exist an admissible graph and a family of admissible multi-fans describing $M$. The rest of the theorem follows from Theorem~\ref{tt33}.
\end{proof}

We have further established the following relationships.
\begin{center}
$\{\textrm{4-dimensional almost complex torus manifolds}\} \to \{\textrm{admissible graphs}\} \iff \{\textrm{families of admissible multi-fans}\} \to \{\textrm{4-dimensional almost complex torus manifolds}\}$
\end{center}

\section{Blow up and down of manifold, graph, and multi-fan} \label{s5.2}

In this section, we describe equivariant blow up and blow down of a 4-dimensional almost complex torus manifold and prove Proposition \ref{tp43}, which demonstrates that blow up and blow down of the manifold, its family of multi-fans, and its graph all correspond to each other.

Let $k \in \{1,2\}$. Let a torus $T^k$ act on $\mathbb{C}^2$ by
\begin{center}
$g \cdot (z_1,z_2)=(g^{w_1} z_1,g^{w_2} z_2)$
\end{center}
for all $g \in T^k \subset \mathbb{C}^k$, for some non-zero $w_1$ and $w_2$ in $\mathbb{Z}^k$ such that $w_1 \neq w_2$. The origin $(0,0)$ is a fixed point of the action. The \textbf{blow up} of the origin is an operation that replaces the origin in $\mathbb{C}^2$ by the set of complex straight lines passing through it. We can describe the blown up space $\widetilde{\mathbb{C}^2}$ as
\begin{center}
$\widetilde{\mathbb{C}^2}=\{(z,l) \, | \, z \in l\} \subset \mathbb{C}^2 \times \mathbb{CP}^1$.
\end{center}
Alternatively, we can describe $\widetilde{\mathbb{C}^2}$ using the equation
\begin{center}
$\widetilde{\mathbb{C}^2}=\{((z_1,z_2),[y_1:y_2]) \, | \, y_1 z_2=y_2 z_1\}$.
\end{center}
The $T^k$-action on $\mathbb{C}^2$ naturally extends to a $T^k$-action on $\widetilde{\mathbb{C}^2}$ by
\begin{center}
$g \cdot ((z_1,z_2),[y_1:y_2])=((g^{w_1}z_1,g^{w_2} z_2), [g^{w_1}y_1:g^{w_2} y_2])$
\end{center}
for all $g \in T^k$. This action of $T^k$ on $\widetilde{\mathbb{C}^2}$ has two fixed points $q'=((0,0),[1:0])$ and $q''=((0,0),[0:1])$. Using local coordinates $(z_1,\frac{y_2}{y_1})$ near $q'$, the weights at $q'$ are $\{w_1,w_2-w_1\}$. Similarly, using local coordinates $(z_2,\frac{y_1}{y_2})$, the weights at $q''$ are $\{w_2,w_1-w_2\}$. The blown up space $\widetilde{\mathbb{C}^2}$ is the holomorphic line bundle $\mathcal{O}(-1)$ over $\mathbb{CP}^1$, whose self intersection number of the zero section is $-1$.

The \textbf{blow down} is the reverse operation $\pi:\widetilde{\mathbb{C}^2} \to \mathbb{C}^2$ of the blow up of the origin in $\mathbb{C}^2$. We shall show that for an action of a torus $T^k$ on a 4-dimensional almost complex manifold $M$, we can equivariantly blow down a 2-sphere $\mathbb{CP}^1$ if two fixed points $p'$ and $p''$ have weights $\{w_1,w_2-w_1\}$ and $\{w_2,w_1-w_2\}$, respectively, and the $T^k$-action on $M$ restricts to act on this $\mathbb{CP}^1$ by 
\begin{center}
$g \cdot [y_1:y_2]=[g^{w_1}y_1:g^{w_2}y_2]$, 
\end{center}
fixing $p_1$ and $p_2$ with weights $w_2-w_1$ and $w_1-w_2$; in other words, $(p_1,p_2)$ is $(w_2-w_1)$-sphere. We can blow down this $2$-sphere, since it has self-intersection number $-1$.

Let a torus $T^k$ act on a 4-dimensional almost complex manifold $M$. Suppose that an isolated fixed point $p$ has weights $\{w_1,w_2\}$ for some non-zero $w_1,w_2 \in \mathbb{Z}^k$ such that $w_1 \neq w_2$. 
The tangent space $T_pM$ of $M$ at $p$ is identified with $\mathbb{C}^2$ on which the $T^k$ acts by $g \cdot (z_1,z_2)=(g^{w_1}z_1,g^{w_2}z_2)$ for all $g \in T^k$ and $(z_1,z_2) \in \mathbb{C}^2$. 
Suppose that the almost complex structure is integrable near $p$.
By the local linearization theorem, a neighborhood of $p$ is $T^k$-equivariantly biholomorphic to a neighborhood of 0 in $T_pM \simeq \mathbb{C}^2$ with this action. 
Therefore, we can equivariantly blow up $p$ by identifying a neighborhood of $p$ with a neighborhood of the origin in $\mathbb{C}^2$, on which $T^k$ acts by $g \cdot (z_1,z_2)=(g^{w_1} z_1,g^{w_2} z_2)$. The blow up replaces the fixed point $p$ with a 2-sphere $\mathbb{CP}^1=\{[y_1:y_2]\}$, and a neighborhood of this 2-sphere in the blown up manifold is equivariantly biholomorphic to $\widetilde{\mathbb{C}^2}$ with an action
\begin{center}
$g \cdot ((z_1,z_2),[y_1:y_2])=((g^{w_1}z_1,g^{w_2} z_2), [g^{w_1}y_1:g^{w_2} y_2])$.
\end{center}
In particular, $T^k$ acts on this 2-sphere $\mathbb{CP}^1=\{((0,0),([y_1:y_2])\}$ by
\begin{center}
$g \cdot ((0,0),[y_1:y_2])=((0,0),[g^{w_1} y_1:g^{w_2}y_2])$ 
\end{center}
with fixed points $p'=((0,0),[1:0])$ and $p''=((0,0),[0:1])$, which have weights $w_2-w_1$ and $w_1-w_2$ for this action, respectively. Let $\widetilde{M}$ denote the blown up manifold. Then $(p',p'')$ is $(w_2-w_1)$-sphere in $\widetilde{M}$ (see Definition \ref{tsphere}). The $T^k$-action on $M$ extends to act on $\widetilde{M}$ so that the weights at $p'$ and $p''$ are $\{w_1,w_2-w_1\}$ and $\{w_2,w_1-w_2\}$, respectively. 

Suppose further that $M$ is compact and the fixed point set of $M$ is discrete. If a graph $\Gamma$ describes $M$, then $\Gamma'$ describes the blown up manifold $M'$, where $\Gamma'$ is obtained from $\Gamma$ by blowing up $p$, which replaces the vertex $p$ with $(p',p'')$-edge with label $(w_2-w_1)$. Therefore, the following lemma holds.

\begin{lem} \emph{\textbf{[Blow up]}} \label{tl41} Let a torus $T^k$ act effectively on a 4-dimensional almost complex manifold $M$. Suppose that a fixed point $p$ has weights $\{w_1,w_2\}$ for some non-zero $w_1$ and $w_2$ in $\mathbb{Z}^k$ such that $w_1 \neq w_2$.
Suppose that the almost complex structure is integrable near $p$.
Then we can equivariantly blow up $p$ to construct another 4-dimensional almost complex manifold $M'$ equipped with a $T^k$ action in which $p$ is replaced with $(w_2-w_1)$-sphere $(p',p'')$. The weights at $p'$ and $p''$ are $\{w_1,w_2-w_1\}$ and $\{w_2,w_1-w_2\}$, respectively. 

Suppose that $M$ is compact and has a discrete fixed point set. Let $\Gamma$ be a multigraph describing $M$, and let $\Gamma'$ be a multigraph obtained from $\Gamma$ by blowing up $p$. Then $\Gamma'$ describes $M'$.
\end{lem}

The next lemma explains when we can blow down a 4-dimensional almost complex $T^k$-manifold.

\begin{lem} \emph{\textbf{[Blow down]}} \label{tl42}
Let a torus $T^k$ act on a 4-dimensional compact almost complex manifold $M$. Suppose that two fixed points $p'$ and $p''$ have weights $\{w_1,w_2-w_1\}$ and $\{w_1-w_2,w_2\}$, respectively, for some non-zero $w_1$ and $w_2$ in $\mathbb{Z}^k$ with $w_1 \neq w_2$, and that $(p',p'')$ is $(w_2-w_1)$-sphere. Assume that the almost complex structure is integrable near the $(w_2-w_1)$-sphere $(p',p'')$. Then we can equivariantly blow down the 2-sphere $(p',p'')$ to construct another 4-dimensional compact almost complex manifold $M'$ equipped with a $T^k$-action, in which the $(w_2-w_1)$-sphere $(p',p'')$ is shrunken to a fixed point $p$ that has weights $\{w_1,w_2\}$. \end{lem}

\begin{proof}
Let $F$ denote the $(w_2-w_1)$-sphere $(p',p'')$. By assumption, the almost complex structure is integrable near $F$. A $T^k$-invariant neighborhood of $F$ in $M$ is equivariantly biholomorphic to the normal bundle $NF$ of $F$, and the normal bundle $NF$ is equivariantly biholomorphic to a $T^k$-invariant neighborhood of $\mathbb{CP}^1$ in $\widetilde{\mathbb{C}^2}$, where on $\widetilde{\mathbb{C}^2}$ the torus $T^k$ acts by
\begin{center}
$g \cdot ((z_1,z_2),[y_1:y_2])=((g^{w_1}z_1,g^{w_2} z_2), [g^{w_1}y_1:g^{w_2} y_2])$.
\end{center}
as described in the beginning of this section. Thus, we can equivariantly blow down the 2-sphere $F$ in $M$ containing $p'$ and $p''$ to a fixed point $p$ that has weights $\{w_1,w_2\}$. \end{proof}

We show that blow up and blow down of a 4-dimensional almost complex torus manifold $M$, those of a family of multi-fans describing $M$, and those of an admissible graph describing $M$ all correspond to each other.

\begin{pro} \label{tp43}
Let $M$ be a 4-dimensional almost complex torus manifold, let $\Delta$ be a family of multi-fans describing $M$, and let $\Gamma$ be a graph describing $M$. 
\begin{enumerate}
\item Suppose that the almost complex structure is integrable near a fixed point $p$. Blow up of $p$ in $M$ in the sense of Lemma \ref{tl41}, blow up of $\Delta$, and blow up of $\Gamma$ all correspond to each other.
\item Suppose that the almost complex structure is integrable near $w$-sphere. Blow down of $w$-sphere of $M$ in the sense of Lemma \ref{tl42}, blow down of $\Delta$, and blow down of $\Gamma$ all correspond to each other.
\end{enumerate}
\end{pro}

\begin{proof}

\begin{figure}
\centering
\begin{subfigure}[b][4.8cm][s]{.5\textwidth}
\centering
\begin{tikzpicture}
\draw[->] (-1, 0) -- (3, 0) node[right] {$x$};
\draw[->] (0, -1) -- (0, 3) node[above] {$y$};
\draw[->] (0, 0) -- (2, 1) node[right] {$w_{i-1}$};
\draw[->] (0, 0) -- (1, 1) node[above] {$w_i$};
\end{tikzpicture}
\caption{Before: $\Delta$}\label{tfig2-1}
\end{subfigure}
\begin{subfigure}[b][4.8cm][s]{.4\textwidth}
\centering
\begin{tikzpicture}
\draw[->] (-1, 0) -- (3, 0) node[right] {$x$};
\draw[->] (0, -1) -- (0, 3) node[above] {$y$};
\draw[->] (0, 0) -- (2, 1) node[right] {$w_{i-1}$};
\draw[->] (0, 0) -- (1, 1) node[above] {$w_i$};
\draw[->] (0, 0) -- (3, 2) node[above] {$w_{i-1}+w_{i}$};
\end{tikzpicture}
\caption{Blow up $(w_{i-1},w_i)$: $\Delta'$}\label{tfig2-2}
\end{subfigure}\qquad
\caption{Blow up of a multi-fan}\label{tfig2}
\end{figure}
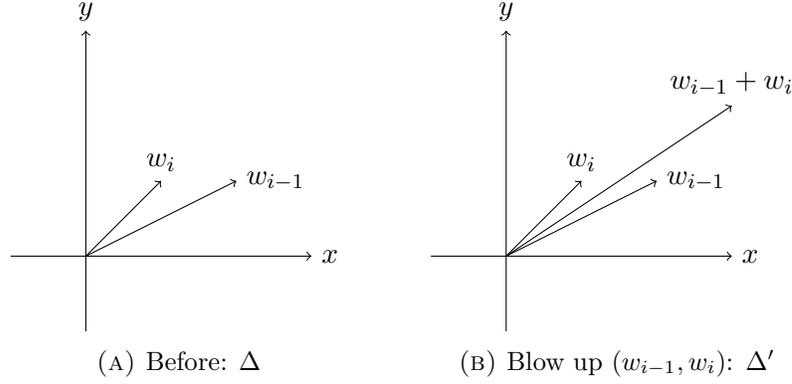

\begin{figure}
\centering
\begin{subfigure}[b][6.8cm][s]{.4\textwidth}
\centering
\vfill
\begin{tikzpicture}[state/.style ={circle, draw}]
\node[state] (a) {$p_{i-1}$};
\node[state] (b) [above right=of a] {$p_i$};
\node[state] (c) [above left=of b] {$p_{i+1}$};
\path (a) [->] edge node[right] {$w_{i-1}$} (b);
\path (b) [->] edge node [right] {$w_i$} (c);
\end{tikzpicture}
\vfill
\caption{Before: $\Gamma$}\label{tfig3-1}
\end{subfigure}
\begin{subfigure}[b][6.8cm][s]{.4\textwidth}
\centering
\begin{tikzpicture}[state/.style ={circle, draw}]
\node[state] (a) {$p_{i-1}$};
\node[state] (b) [above right=of a] {$p_i'$};
\node[state] (c) [above=of b] {$p_i''$};
\node[state] (d) [above left=of c] {$p_{i+1}$};
\path (a) [->] edge node[right] {$w_{i-1}$} (b);
\path (b) [->] edge node [left] {$w_i+w_{i-1}$} (c);
\path (c) [->] edge node [left] {$w_i$} (d);
\end{tikzpicture}
\caption{Blow up $p_i$; $\Gamma'$}\label{tfig3-2}
\end{subfigure}\qquad
\caption{Blow up of a graph}\label{tfig3}
\end{figure}
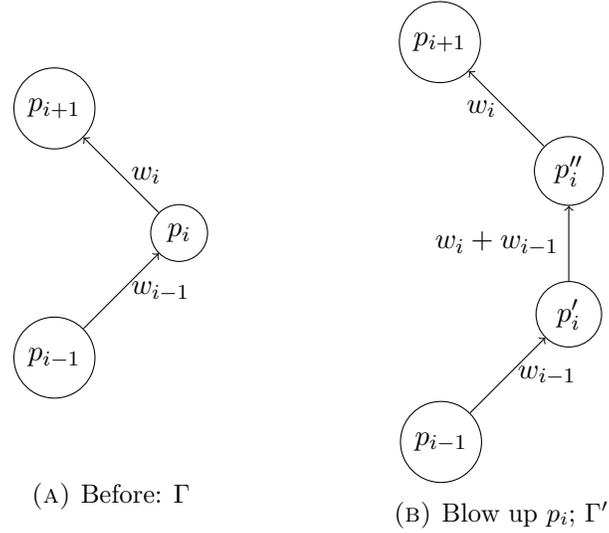

First, we prove (1). Suppose that a fixed point $p_i \in M^{T^2}$ has weights $\{-w_{i-1},w_i\}$. Since $\Gamma$ describes $M$, this means that $(p_{i-1},p_i)$ is $w_{i-1}$-edge and $(p_i,p_{i+1})$ is $w_i$-edge for some vertices (fixed points) $p_{i-1}$ and $p_{i+1}$ of some component $\Gamma_j$ of $\Gamma$. Let $V_j$ be a multi-fan corresponding to $\Gamma_j$ in the sense of Lemma \ref{tl21}.

We blow up $p_i$ in the sense of Lemma \ref{tl41}. Blowing up $p_i$ in $M$ in the sense of Lemma \ref{tl41} replaces $p_i$ with $(w_{i-1}+w_{i})$-sphere $(p_i',p_i'')$. Let $M'$ be the blown up manifold, let $\Delta'$ (Figure \ref{tfig2-2}) be a family of multi-fans obtained from $\Delta$ (Figure \ref{tfig2-1}) by blowing up $(w_{i-1},w_i)$ in the sense of Definition \ref{td112}, and let $\Gamma'$ (Figure \ref{tfig3-2}) be a graph obtained from $\Gamma$ (Figure \ref{tfig3-1}) by blowing up of $p_i$ in the sense of Definition \ref{td113}, respectively. Then $\Delta'$ and $\Gamma'$ each describe $M'$. This proves (1).

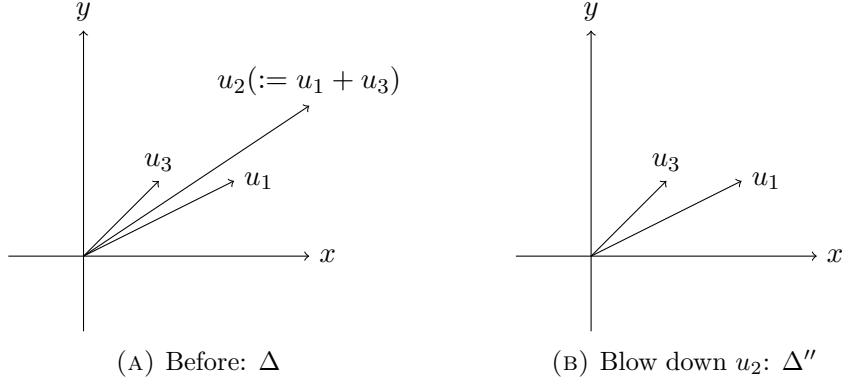
\begin{figure}
\centering
\begin{subfigure}[b][5cm][s]{.4\textwidth}
\centering
\begin{tikzpicture}
\draw[->] (-1, 0) -- (3, 0) node[right] {$x$};
\draw[->] (0, -1) -- (0, 3) node[above] {$y$};
\draw[->] (0, 0) -- (2, 1) node[right] {$u_1$};
\draw[->] (0, 0) -- (1, 1) node[above] {$u_3$};
\draw[->] (0, 0) -- (3, 2) node[above] {$u_2(:=u_1+u_3)$};
\end{tikzpicture}
\caption{Before: $\Delta$}\label{tfig4-1}
\end{subfigure}\qquad
\begin{subfigure}[b][5cm][s]{.5\textwidth}
\centering
\begin{tikzpicture}
\draw[->] (-1, 0) -- (3, 0) node[right] {$x$};
\draw[->] (0, -1) -- (0, 3) node[above] {$y$};
\draw[->] (0, 0) -- (2, 1) node[right] {$u_1$};
\draw[->] (0, 0) -- (1, 1) node[above] {$u_3$};
\end{tikzpicture}
\caption{Blow down $u_2$: $\Delta''$}\label{tfig4-2}
\end{subfigure}
\caption{Blow down of a multi-fan}\label{tfig4}
\end{figure}

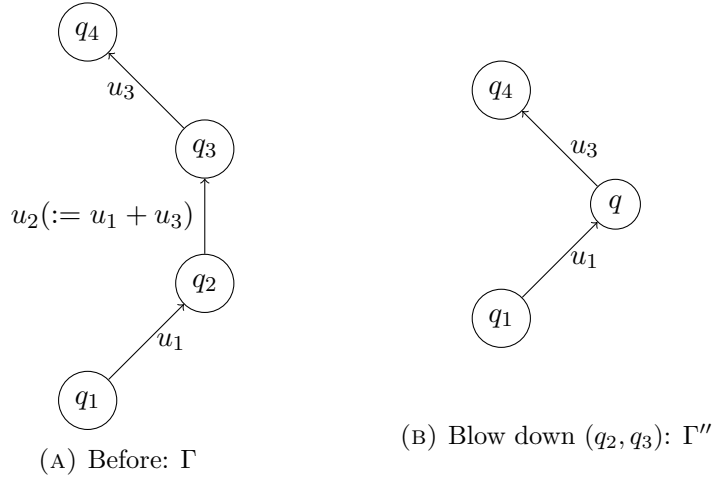
\begin{figure}
\centering
\begin{subfigure}[b][6cm][s]{.4\textwidth}
\centering
\begin{tikzpicture}[state/.style ={circle, draw}]
\node[state] (a) {$q_1$};
\node[state] (b) [above right=of a] {$q_2$};
\node[state] (c) [above=of b] {$q_3$};
\node[state] (d) [above left=of c] {$q_4$};
\path (a) [->] edge node[right] {$u_1$} (b);
\path (b) [->] edge node [left] {$u_2(:=u_1+u_3)$} (c);
\path (c) [->] edge node [left] {$u_3$} (d);
\end{tikzpicture}
\caption{Before: $\Gamma$}\label{tfig5-1}
\end{subfigure}\qquad
\begin{subfigure}[b][6cm][s]{.4\textwidth}
\centering
\vfill
\begin{tikzpicture}[state/.style ={circle, draw}]
\node[state] (a) {$q_1$};
\node[state] (b) [above right=of a] {$q$};
\node[state] (c) [above left=of b] {$q_4$};
\path (a) [->] edge node[right] {$u_1$} (b);
\path (b) [->] edge node [right] {$u_3$} (c);
\end{tikzpicture}
\vfill
\caption{Blow down $(q_2, q_3)$: $\Gamma''$}\label{tfig5-2}
\end{subfigure}
\caption{Blow down of a graph}\label{tfig5}
\end{figure}

Next, we prove (2). To do so, suppose that a fixed point $q_2$ has weights $\{-u_1,u_2\}$, a fixed point $q_3$ has weights $\{-u_2,u_3\}$, and $(q_2,q_3)$ is $u_2$-sphere, where $u_2=u_1+u_3$. By Lemma \ref{tl42}, we can blow down this $u_2$-sphere $(q_2,q_3)$ to a fixed point $q$ that has weights $\{-u_1,u_3\}$. Let $M''$ be the blown down manifold, let $\Delta''$ (Figure \ref{tfig4-2}) be a family of multi-fans obtained from $\Delta$ (Figure \ref{tfig4-1}) by blowing down of $u_2$ in the sense of Definition \ref{td112}, and let $\Gamma''$ (Figure \ref{tfig5-2}) be a graph obtained from $\Gamma$ (Figure \ref{tfig5-1}) by blowing down of the edge $(q_2,q_3)$ in the sense of Definition \ref{td113}, respectively. Then $\Delta''$ and $\Gamma''$ each describe $M''$. This proves (2). \end{proof}

\section{Proof of Theorem \ref{tt114} and Theorem \ref{tt114a}} \label{s5.3}

In this section, we prove Theorem \ref{tt114} and Theorem \ref{tt114a}. Theorem \ref{tt114} provides a minimal model and operations for each of a family of multi-fans and a graph describing a 4-dimensional almost complex torus manifold. We need the following simple fact.

\begin{lemma} \label{tl51}
Let $v_1$ and $v_2$ be a basis of $\mathbb{Z}^2$ such that $|v_1|<|v_2|$. Then either $|v_2-v_1|<|v_2|$ or $|v_2+v_1|<|v_2|$.
\end{lemma}

\begin{proof}
We claim that either the angle between $v_1$ and $v_2$ is at most $\pi/4$ or the angle between $-v_1$ and $v_2$ is at most $\pi/4$. To see this, suppose on the contrary that both angles are bigger than $\pi/4$. Since $\max\{|v_1|^2,|v_2|^2\} \geq 2$, a parallelogram formed by $v_1$ and $v_2$ at $(0,0)$ contains a lattice point in its interior, which contradicts the assumption that $v_1$ and $v_2$ form a basis of $\mathbb{Z}^2$.

First, suppose that the angle $\theta$ between $v_1$ and $v_2$ is at most $\pi/4$. Then we have $|v_1|^2 < 2|v_1| \, |v_2| \cos \theta$. By the cosine rule $|v_2-v_1|^2=|v_2|^2+|v_1|^2-2|v_1| \, |v_2|\cos \theta$, it follows that $|v_2-v_1| < |v_2|$.

Second, suppose that the angle $\theta$ between $-v_1$ and $v_2$ is at most $\pi/4$. Then $|-v_1|^2 < 2|-v_1| \, |v_2| \cos \theta$. By the cosine rule $|v_2+v_1|^2=|v_2|^2+|-v_1|^2-2|-v_1| \, |v_2|\cos \theta$, it follows that $|v_2+v_1| < |v_2|$. \end{proof}

We are now ready to prove Theorem \ref{tt114}.

\begin{proof}[\textbf{Proof of Theorem \ref{tt114}}]
Let $(M,J)$ be a 4-dimensional almost complex torus manifold.
By Proposition \ref{tp24}, there exist an admissible graph $\Gamma$ and a family $\Delta$ of admissible multi-fans describing $M$.
Because blow up and blow down of $\Gamma$ and $\Delta$ correspond to each other, it suffices to prove this theorem for graphs.
Since $\Gamma$ describes $M$, by Proposition \ref{tp23}, if $(p,q)$ is $w$-edge of $\Gamma$, then $(p,q)$ is $w$-sphere.

Suppose that $\max\{|w_{pi}|:p \in M^{T^2}, 1 \leq i \leq 2\}=1$. Then $\Delta$ is a family of minimal multi-fans and $\Gamma$ is a minimal graph, and thus this theorem follows.
Therefore, from now on, we suppose that $\max\{|w_{pi}|:p \in M^{T^2}, 1 \leq i \leq 2\}>1$.

Suppose that $(p_1,p_2)$ is $w$-edge, where $|w|=\max\{|w_{pi}|:p \in M^{T^2}, 1 \leq i \leq 2\}>1$. Let $(p',p_1)$ be $w_1$-edge and let $(p_2,p'')$ be $w_2$-edge for some fixed points $p'$ and $p''$, and for some $w_1,w_2 \in \mathbb{Z}^2$. Then $p_1$ has weights $\{-w_1,w\}$ and $p_2$ has weights $\{-w,w_2\}$. 

Suppose that $|w_1|=|w|$. Then $|w_1|=|w|>1$ and hence a parallelogram formed by $w_1$ and $w$ at $(0,0)$ contains a lattice point in its interior. This contradicts that $w_1$ and $w$ form a basis of $\mathbb{Z}^2$. Therefore, $|w_1|<|w|$. Similarly, $|w_2|<|w|$.

Since $(p_1,p_2)$ is $w$-sphere, $p_1$ has weights $\{-w_1,w\}$, and $p_2$ has weights $\{-w,w_2\}$, by Lemma \ref{tl23}, we have $-w_1=w_2+aw$, where $a$ is the self-intersection number of the $w$-sphere $(p_1,p_2)$. Since $|w_1|<|w|$ and $|w_2|<|w|$, the equation $-w_1=w_2+aw$ implies that $-1 \leq a \leq 1$.

\begin{figure}
\centering
\begin{subfigure}[b][6cm][s]{.45\textwidth}
\centering
\vfill
\begin{tikzpicture}[state/.style ={circle, draw}]
\node[state] (a) {$p'$};
\node[state] (b) [above right=of a] {$p_1$};
\node[state] (c) [above=of b] {$p_2$};
\node[state] (d) [above left=of c] {$p''$};
\path (a) [->] edge node[left] {$w_1$} (b);
\path (b) [->] edge node [left] {$w_1+w_2$} (c);
\path (c) [->] edge node [left] {$w_2$} (d);
\end{tikzpicture}
\vfill
\caption{Before: $\Gamma$}\label{tfig6-1}
\end{subfigure}
\begin{subfigure}[b][4.5cm][s]{.45\textwidth}
\centering
\vfill
\begin{tikzpicture}[state/.style ={circle, draw}]
\node[state] (a) {$p'$};
\node[state] (b) [above right=of a] {$p$};
\node[state] (c) [above left=of b] {$p''$};
\path (a) [->] edge node[right] {$w_1$} (b);
\path (b) [->] edge node [right] {$w_2$} (c);
\end{tikzpicture}
\vfill
\caption{Blow down $(p_1,p_2)$ in $M$ to $M'$: $\Gamma'$}\label{tfig6-2}
\end{subfigure}\qquad
\caption{Case (1): $a=-1$}\label{tfig6}
\end{figure}

First, suppose that $a=-1$. We have that $p_1$ has weights $\{-w_1,w_1+w_2\}$, $p_2$ has weights $\{-w_1-w_2,w_2\}$, and $(p_1,p_2)$ is $(w_1+w_2)$-sphere; see Figure \ref{tfig6-1}. 
Let $\Gamma'$ be a graph obtained from $\Gamma$ by blow down the $(p_1,p_2)$-edge; see Figure \ref{tfig6-2}.
Since $\Gamma'$ is admissible, by Theorem~\ref{tt33}, there is a 4-dimensional almost complex torus manifold $M'$ described by $\Gamma'$; in the case that $J$ is integrable near the $(w_1+w_2)$-sphere $(p_1,p_2)$, we take $M'$ to be blow down of the $(w_1+w_2)$-sphere $(p_1,p_2)$; by Proposition~\ref{tp43}, $\Gamma'$ describes $M'$.

Second, suppose that $a=0$. In this case, we have that $w_2=-w_1$, $p_1$ has weights $\{-w_1, w\}$, $p_2$ has weights $\{-w, -w_1\}$, and $(p_1,p_2)$ is $w$-sphere. By Lemma \ref{tl51}, either (i) $|w-w_1|<|w|$ or (ii) $|w+w_1|<|w|$.

\begin{figure}
\centering
\begin{subfigure}[b][5.8cm][s]{.4\textwidth}
\centering
\vfill
\begin{tikzpicture}[state/.style ={circle, draw}]
\node[state] (b) {$p_2$};
\node[state] (c) [below right=of b] {$p_1$};
\node[state] (d) [above right=of c] {$p'$};
\node[state] (a) [above right=of b] {$p''$};
\path (c) [->] edge node [right] {$w$} (b);
\path (d) [->] edge node [right] {$w_1$} (c);
\path (b) [->] edge node [right] {$-w_1$} (a);
\end{tikzpicture}
\vfill
\caption{Before: $\Gamma$}\label{tfig7-1}
\end{subfigure}\qquad
\begin{subfigure}[b][5.8cm][s]{.4\textwidth}
\centering
\begin{tikzpicture}[state/.style ={circle, draw}]
\node[state] (b) {$p_2''$};
\node[state] (c) [below=of b] {$p_2'$};
\node[state] (d) [below right=of c] {$p_1$};
\node[state] (e) [above right=of d] {$p'$};
\node[state] (a) [above right=of b] {$p''$};
\path (c) [->] edge node [right] {$w-w_1$} (b);
\path (d) [->] edge node [left] {$w$} (c);
\path (e) [->] edge node [right] {$w_1$} (d);
\path (b) [->] edge node [right] {$-w_1$} (a);
\end{tikzpicture}
\caption{Blow up $p_2$: $\Gamma_1$}\label{tfig7-2}
\end{subfigure}\qquad
\begin{subfigure}[b][4.5cm][s]{.4\textwidth}
\centering
\vfill
\begin{tikzpicture}[state/.style ={circle, draw}]
\node[state] (b) {$p_2''$};
\node[state] (c) [below right=of b] {$p_1'$};
\node[state] (d) [above right=of c] {$p'$};
\node[state] (a) [above right=of b] {$p''$};
\path (c) [->] edge node [left] {$w-w_1$} (b);
\path (d) [->] edge node [right] {$w_1$} (c);
\path (b) [->] edge node [left] {$-w_1$} (a);
\end{tikzpicture}
\caption{Blow down $(p_1,p_2')$: $\Gamma'$}\label{tfig7-3}
\end{subfigure}\qquad
\caption{Case (2-i): $a=0$ and $|w-w_1|<|w|$}\label{tfig7}
\end{figure}
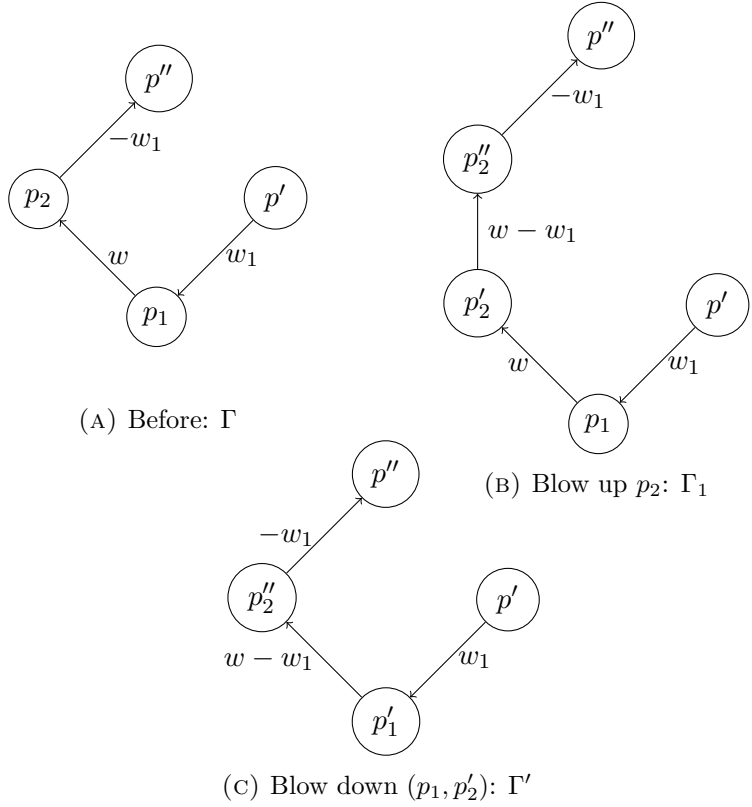

Assume that (i) $|w-w_1|<|w|$. Figure \ref{tfig7-1} illustrates $\Gamma$.
Let $\Gamma_1$ (Figure \ref{tfig7-2}) be a graph obtained from $\Gamma$ by blow up $p_2$ to $(w-w_1)$-edge $(p_2',p_2'')$; then $\Gamma_1$ is admissible, and hence by Theorem~\ref{tt33}, $\Gamma_1$ describes some 4-dimensional almost complex torus manifold $M_1$.
Next, in $M_1$, we have that $p_1$ has weights $\{-w_1,w\}$, $p_2'$ has weights $\{-w,w-w_1\}$, and $(p_1,p_2')$ is $w$-sphere.
Let $\Gamma'$ (Figure \ref{tfig7-3}) be a graph obtained from $\Gamma_1$ by blow down the $w$-edge $(p_1,p_2')$ to a vertex $p_2''$; then $\Gamma'$ describes some almost complex torus manifold $M'$ by Theorem~\ref{tt33}.
These steps replaced $w$-sphere $(p_1,p_2)$ with $(w-w_1)$-sphere $(p_1',p_2'')$, where $|w-w_1|<|w|$. Figure \ref{tfig7} illustrates these steps.
In the case that $J$ is integrable near each $w$-sphere, we take $M_1$ to be blow up of $p_2$ in $M$ and $M'$ to be blow down of $w$-sphere $(p_1,p_2')$ in $M_1$.

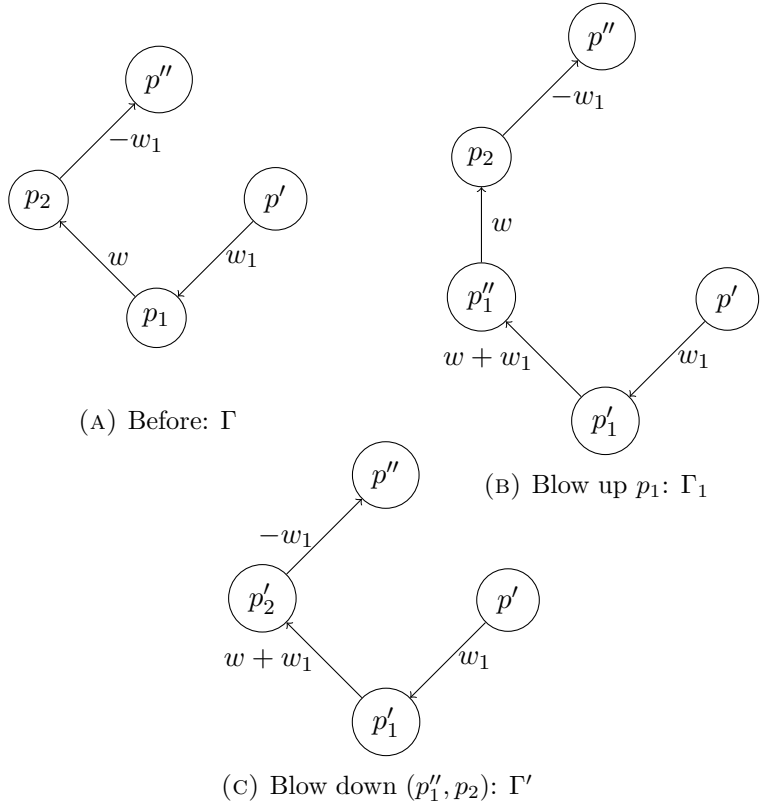
\begin{figure}
\centering
\begin{subfigure}[b][5.8cm][s]{.4\textwidth}
\centering
\vfill
\begin{tikzpicture}[state/.style ={circle, draw}]
\node[state] (b) {$p_2$};
\node[state] (c) [below right=of b] {$p_1$};
\node[state] (d) [above right=of c] {$p'$};
\node[state] (a) [above right=of b] {$p''$};
\path (c) [->] edge node [right] {$w$} (b);
\path (d) [->] edge node [right] {$w_1$} (c);
\path (b) [->] edge node [right] {$-w_1$} (a);
\end{tikzpicture}
\vfill
\caption{Before: $\Gamma$}\label{tfig8-1}
\end{subfigure}\qquad
\begin{subfigure}[b][5.8cm][s]{.4\textwidth}
\centering
\begin{tikzpicture}[state/.style ={circle, draw}]
\node[state] (b) {$p_2$};
\node[state] (c) [below=of b] {$p_1''$};
\node[state] (d) [below right=of c] {$p_1'$};
\node[state] (e) [above right=of d] {$p'$};
\node[state] (a) [above right=of b] {$p''$};
\path (c) [->] edge node [right] {$w$} (b);
\path (d) [->] edge node [left] {$w+w_1$} (c);
\path (e) [->] edge node [right] {$w_1$} (d);
\path (b) [->] edge node [right] {$-w_1$} (a);
\end{tikzpicture}
\caption{Blow up $p_1$: $\Gamma_1$}\label{tfig8-2}
\end{subfigure}\qquad
\begin{subfigure}[b][4.5cm][s]{.4\textwidth}
\centering
\begin{tikzpicture}[state/.style ={circle, draw}]
\node[state] (b) {$p_2'$};
\node[state] (c) [below right=of b] {$p_1'$};
\node[state] (d) [above right=of c] {$p'$};
\node[state] (a) [above right=of b] {$p''$};
\path (c) [->] edge node [left] {$w+w_1$} (b);
\path (d) [->] edge node [right] {$w_1$} (c);
\path (b) [->] edge node [left] {$-w_1$} (a);
\end{tikzpicture}
\caption{Blow down $(p_1'',p_2)$: $\Gamma'$}\label{tfig8-3}
\end{subfigure}\qquad
\caption{Case (2-ii): $a=0$ and $|w+w_1|<|w|$}\label{tfig8}
\end{figure}

Assume that (ii) $|w+w_1|<|w|$. Figure \ref{tfig8-1} illustrates $\Gamma$. Let $\Gamma_1$ (Figure \ref{tfig8-2}) be a graph obtained from $\Gamma$ by blow up $p_1$ to $(w+w_1)$-edge $(p_1',p_1'')$; then $\Gamma_1$ is admissible, and hence by Theorem~\ref{tt33}, $\Gamma_1$ describes some 4-dimensional almost complex torus manifold $M_1$. Next, in $M_1$, we have that $p_1''$ has weights $\{-w-w_1,w\}$, $p_2$ has weights $\{-w,-w_1\}$, and $(p_1'',p_2)$ is $w$-sphere. Let $\Gamma'$ (Figure \ref{tfig8-3}) be a graph obtained from $\Gamma_1$ by blow down the $w$-edge $(p_1'',p_2)$ to a vertex $p_2'$; then $\Gamma'$ describes some almost complex torus manifold $M'$ by Theorem~\ref{tt33}. These steps replaced $w$-sphere $(p_1,p_2)$ with $(w+w_1)$-sphere $(p_1',p_2')$, where $|w+w_1|<|w|$. Figure \ref{tfig8} illustrates these steps.
In the case that $J$ is integrable near each $w$-sphere, we take $M_1$ to be blow up of $p_1$ in $M$ and $M'$ to be blow down of $w$-sphere $(p_1'',p_2)$ in $M_1$.

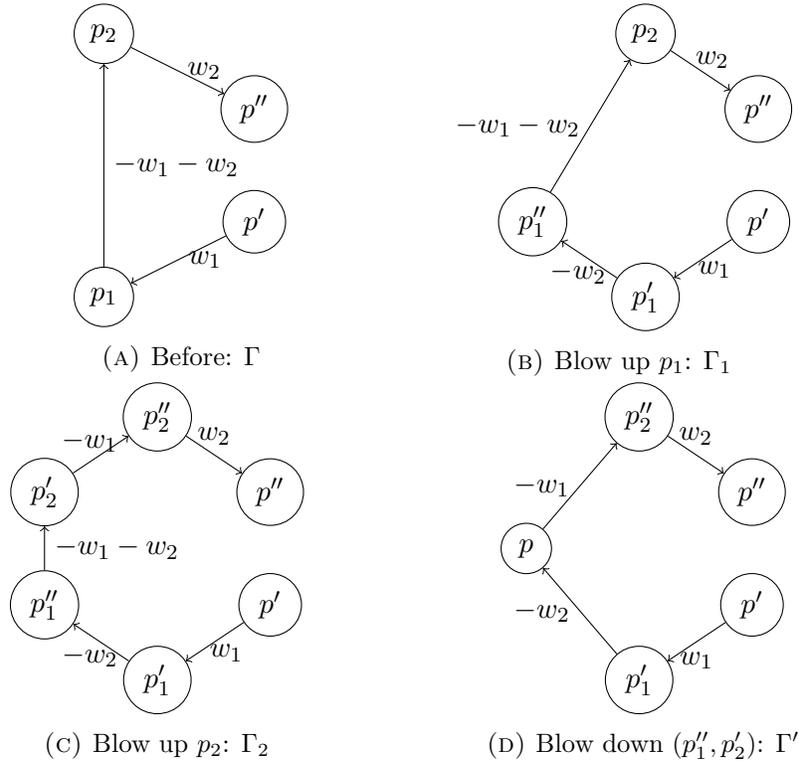
\begin{figure}
\centering
\begin{subfigure}[b][5cm][s]{.45\textwidth}
\centering
\begin{tikzpicture}[state/.style ={circle, draw}]
\node[vertex] (a) at (0,1) {$p_2$};
\node[vertex] (b) at (2,0) {$p''$};
\node[vertex] (c) at (2,-1.5) {$p'$};
\node[vertex] (d) at (0,-2.5) {$p_1$};
\path (a) [->] edge node[right] {$w_2$} (b);
\path (d) [->] edge node [right] {$-w_1-w_2$} (a);
\path (c) [->] edge node [right] {$w_1$} (d);
\end{tikzpicture}
\caption{Before: $\Gamma$}\label{tfig9-1}
\end{subfigure}
\begin{subfigure}[b][5cm][s]{.45\textwidth}
\centering
\begin{tikzpicture}[state/.style ={circle, draw}]
\node[vertex] (a) at (0,1) {$p_2$};
\node[vertex] (b) at (1.5,0) {$p''$};
\node[vertex] (c) at (1.5,-1.5) {$p'$};
\node[vertex] (d) at (0,-2.5) {$p_1'$};
\node[vertex] (e) at (-1.5,-1.5) {$p_1''$};
\path (a) [->] edge node[pos=0.7, above] {$w_2$} (b);
\path (d) [->] edge node [pos=0.7, below] {$-w_2$} (e);
\path (e) [->] edge node [left] {$-w_1-w_2$} (a);
\path (c) [->] edge node [pos=0.3, below] {$w_1$} (d);
\end{tikzpicture}
\caption{Blow up $p_1$: $\Gamma_1$}\label{tfig9-2}
\end{subfigure}\qquad
\begin{subfigure}[b][4.8cm][s]{.45\textwidth}
\centering
\begin{tikzpicture}[state/.style ={circle, draw}]
\node[vertex] (a) at (0,1) {$p_2''$};
\node[vertex] (b) at (1.5,0) {$p''$};
\node[vertex] (c) at (1.5,-1.5) {$p'$};
\node[vertex] (d) at (0,-2.5) {$p_1'$};
\node[vertex] (e) at (-1.5,-1.5) {$p_1''$};
\node[vertex] (f) at (-1.5,0) {$p_2'$};
\path (a) [->] edge node[above] {$w_2$} (b);
\path (d) [->] edge node [pos=0.7, below] {$-w_2$} (e);
\path (e) [->] edge node [right] {$-w_1-w_2$} (f);
\path (c) [->] edge node [pos=0.3, below] {$w_1$} (d);
\path (f) [->] edge node [pos=0.3, above] {$-w_1$} (a);
\end{tikzpicture}
\caption{Blow up $p_2$: $\Gamma_2$}\label{tfig9-3}
\end{subfigure}\qquad
\begin{subfigure}[b][4.8cm][s]{.45\textwidth}
\centering
\begin{tikzpicture}[state/.style ={circle, draw}]
\node[vertex] (a) at (0,1) {$p_2''$};
\node[vertex] (b) at (1.5,0) {$p''$};
\node[vertex] (c) at (1.5,-1.5) {$p'$};
\node[vertex] (d) at (0,-2.5) {$p_1'$};
\node[vertex] (e) at (-1.5,-0.75) {$p$};
\path (a) [->] edge node[above] {$w_2$} (b);
\path (d) [->] edge node [left] {$-w_2$} (e);
\path (c) [->] edge node [below] {$w_1$} (d);
\path (e) [->] edge node [left] {$-w_1$} (a);
\end{tikzpicture}
\caption{Blow down $(p_1'',p_2')$: $\Gamma'$}\label{tfig9-4}
\end{subfigure}\qquad
\caption{Case (3): $a=+1$}\label{tfig9}
\end{figure}

Third, suppose that $a=+1$. In this case, we have that $w=-w_1-w_2$, $p_1$ has weights $\{-w_1, -w_1-w_2\}$, $p_2$ has weights $\{w_1+w_2,w_2\}$, and $(p_1,p_2)$ is $(-w_1-w_2)$-sphere; see Figure \ref{tfig9-1}.

Let $\Gamma_1$ (Figure \ref{tfig9-2}) be a graph obtained from $\Gamma$ by blow up $p_1$ to a $(-w_2)$-sphere $(p_1',p_1'')$; then $\Gamma_1$ is admissible, and hence by Theorem~\ref{tt33}, $\Gamma_1$ describes some 4-dimensional almost complex torus manifold $M_1$.
Next, let $\Gamma_2$ (Figure \ref{tfig9-3}) be a graph obtained from $\Gamma_1$ by blow up $p_2$ to a $w_1$-sphere $(p_2',p_2'')$; since $\Gamma_2$ is admissible, by Theorem~\ref{tt33}, $\Gamma_2$ describes some almost complex torus manifold $M_2$.
In $M_2$ ($\Gamma_2$) we have that $(p_1'',p_2')$ is $(-w_1-w_2)$-sphere, $p_1''$ has weights $\{w_2,-w_1-w_2\}$, and $p_2'$ has weights $\{w_1+w_2, -w_1\}$.
Let $\Gamma'$ (Figure \ref{tfig9-4})  be a graph obtained from $\Gamma_2$ by blow down $(-w_1-w_2)$-edge $(p_1'',p_2')$ to a vertex $p$; then $\Gamma'$ is admissible and describes some almost complex torus manifold $M'$, by Theorem~\ref{tt33}.

In the case that $J$ is integrable near each $w$-sphere, we take $M_1$ to be blow up of $p_1$ in $M$, $M_2$ to be blow up of $p_2$ in $M_1$, and $M'$ to be blow down of $(-w_1-w_2)$-sphere $(p_1'',p_2)$ in $M_2$.

We continue the above arguments. By the steps above, we can reduce the magnitude of every weight to 1. That is, we can blow up and blow down $\Gamma$ to another graph $\Gamma''$ describing almost complex torus manifold $M''$, whose labels are all unit vectors in $\mathbb{Z}^2$. That is, $\Gamma''$ is a minimal graph and all weights at the fixed points of $M''$ are unit vectors.

In the case that $J$ is integrable near each $w$-sphere, blow up (down) of $M$, $\Delta$, and $\Gamma$ all correspond to each other by Proposition~\ref{tp43}, and $M$ and $M'$ in each of the above cases are obtained from each other by blow up and blow down, and hence so are $M$ and $M''$.
\end{proof}

Next, we prove Theorem \ref{tt114a}, which demonstrates the existence of a 4-dimensional almost complex torus described by a family of multi-fans and a graph, each obtained from our minimal model (a family of minimal multi-fans and a minimal graph) and operations (blow up and down).

\begin{proof}[\textbf{Proof of Theorem \ref{tt114a}}]
Let $\Delta$ be a family of multi-fans obtained from a family of minimal multi-fans by blow up and down. Since any minimal multi-fan is admissible and blow up and down preserve admissibility, $\Delta$ is admissible. By Theorem \ref{tt33}, there exists a 4-dimensional almost complex torus manifold $M$ described by $\Delta$.

Let $\Gamma$ be a graph obtained from a minimal graph by blow up and down. Any minimal graph is admissible and blow up and down preserve admissibility. Therefore, $\Gamma$ is admissible. By Lemma \ref{tl21}, $\Gamma$ naturally gives rise to a family $\Delta$ of admissible multi-fans. By Theorem \ref{tt33}, there exists a 4-dimensional almost complex torus manifold $M$ described by $\Delta$, and hence by $\Gamma$.
\end{proof}

\section{Correspondence between complex torus 4-manifolds and admissible fans} \label{s5.4}

In this section, we (1) demonstrate that we can equivariantly blow up and down any 4-dimensional complex torus manifold to $\mathbb{CP}^1 \times \mathbb{CP}^1$ with a particular standard action (Theorem \ref{tt115}), (2) show that our fan determines a 4-dimensional complex torus manifold up to equivariant biholomorphism (Corollary \ref{tc117}), and (3) present explicit blow up and down from $M_1$ to $M_2$, where each $M_i$ is $\mathbb{CP}^2$ or one of the Hirzebruch surfaces with particular $T^2$-actions.

Let $M$ be a 4-dimensional compact connected complex surface. Suppose that the circle group acts on $M$ preserving the complex structure and the fixed point set is non-empty and discrete. In \cite{CHK}, Carrell, Kosniowski, and Howard proved that $M$ is obtained from $\mathbb{CP}^2$ or a Hirzebruch surface by blow up.

\begin{theo} \label{t28} \cite{CHK} Let the circle act on a compact connected complex surface $M$ with a non-empty discrete fixed point set. Then $M$ is obtained from $\mathbb{CP}^2$ or a Hirzebruch surface by blow up of fixed points. \end{theo}

For a holomorphic vector field on a complex manifold, Kosniowski proved the following formula, called the Kosniowski formula \cite{Ko}. Hattori and Taniguchi proved the same formula for unitary manifolds \cite{HT}, and Kosniowski and Yahia obtained the same result \cite{KY}. We shall state it for almost complex manifolds.

\begin{theo}[Kosniowski formula] \cite{HT, KY} \label{tt62}
Let the circle act on a compact almost complex manifold $M$. For each connected component $F$ of the fixed point set $M^{S^1}$, let $d(-,F)$ and $d(+,F)$ be the numbers of negative weights and positive weights in the normal bundle $NF$ of $F$, respectively. Then the Hirzebruch $\chi_y$-genus $\chi_y(M)$ of $M$ satisfies
\begin{center}
$\displaystyle \chi_y(M)=\sum_{F \subset M^{S^1}} (-y)^{d(-,F)} \cdot \chi_y(F)=\sum_{F \subset M^{S^1}} (-y)^{d(+,F)} \cdot \chi_y(F)$.
\end{center}
\end{theo}

A family of multi-fans describing a 4-dimensional almost complex torus manifold encodes the Todd genus of the manifold.

\begin{pro} \label{tp63}
Let $M$ be a 4-dimensional almost complex torus manifold. Let $\Delta$ be a family of multi-fans $V_1$, $\cdots$, $V_m$ describing $M$. Then the Todd genus of $M$ is
$\mathrm{Todd}(M)=\sum_{j=1}^m T(V_j)$ (see Definition \ref{td11}).
\end{pro}

\begin{proof}
Let $\xi \in \mathfrak{t} \cap \mathbb{Z}^2$ be an integral element of the Lie algebra $\mathfrak{t}$ of $T^2$ such that the action of the subcircle $S^1$ generated by $\xi$ on $M$ has the same fixed point set as the $T^2$-action on $M$. Such an element $\xi$ exists because there is a finite number of orbit types.

Since $\Delta$ describes $M$, the $T^2$-weights at a fixed point $p_{j,i}$ are $v_{j,i}$ and $-v_{j,i-1}$, for $1 \leq i \leq k_j$ and $1 \leq j \leq m$ (Definition \ref{td13}). Therefore, the $S^1$-weights at $p_{j,i}$ are $v_{j,i}(\xi)$ and $-v_{j,i-1}(\xi)$. In particular, $v_{j,i}(\xi) \neq 0$ for all $i$ and $j$.

Let $\ell$ be a line in $\mathbb{R}^2$ passing through the origin and perpendicular to $\xi$. Let $R_{+}=\{y \in \mathbb{R}^2 \, | \, y(\xi) > 0\}$ and $R_{-}=\{y \in \mathbb{R}^2 \, | \, y(\xi) < 0\}$ be disjoint half-planes in $\mathbb{R}$ separated by $\ell$. Since $v_{j,i}(\xi) \neq 0$ for all $i$ and $j$, every $T^2$-weight $w_{pi}$ at a fixed point $p \in M^{T^2}$ belongs to one of them.

By Theorem \ref{tt62}, the Todd genus of $M$ is equal to the number of fixed points that have no negative $S^1$-weights. On the other hand, a fixed point $p_{j,i}$ has no negative $S^1$-weights if and only if $v_{j,i}(\xi)>0$ and $-v_{j,i-1}(\xi)>0$ if and only if $v_{j,i-1}$ is in $R_-$ and $v_{j,i}$ is in $R_+$. Therefore, each time successive vectors $v_{j,i-1}$ and $v_{j,i}$ pass from $R_-$ to $R_+$, the corresponding fixed point $p_i$ contributes $1$ to the Todd genus of $M$. This implies that the Todd genus of $M$ is precisely the sum $\sum_{j=1}^m T(V_j)$ of numbers $T(V_j)$ of revolutions vectors $v_{j,1}$, $\cdots$, $v_{j,k_j}$ in each multi-fan $V_j$ make around the origin in $\mathbb{R}^2$. \end{proof}

The Todd genus of a 4-dimensional complex torus manifold is 1.

\begin{lem} \label{todd}
Let $M$ be a 4-dimensional complex torus manifold. Then the Todd genus of $M$ is 1.
\end{lem}

\begin{proof}
Let $S^1$ be a subcircle of $\mathbb{T}^2$ generated by $\xi$, where $\xi$ is as in the proof of Proposition \ref{tp63}. Applying Theorem \ref{t28} to the $S^1$-action on $M$, it follows that $M$ is obtained from $\mathbb{CP}^2$ or a Hirzebruch surface by $S^1$-equivariant blow ups at fixed points. The Todd genera of $\mathbb{CP}^2$ and all Hirzebruch surfaces are all 1. Blowing up of a fixed point does not change the Todd genus.
\end{proof}

Therefore, any 4-dimensional complex torus manifold is described by a fan.

\begin{lem} \label{connected}
Let $M$ be a 4-dimensional complex torus manifold. Then there exists a fan describing $M$. Consequently, an admissible graph describing $M$ is connected.
\end{lem}

\begin{proof}
By Lemma \ref{todd}, the Todd genus of $M$ is 1. By Theorem \ref{tt19}, there exists a finite family $\Delta$ of admissible multi-fans $V_1,\cdots,V_m$ describing $M$. By Proposition \ref{tp63}, $\Delta$ consists of a single fan. By Lemma \ref{tl21}, the fan $\Delta$ gives rise to a connected admissible graph describing $M$.
\end{proof}

A 4-dimensional complex torus manifold described by a minimal fan is $\mathbb{CP}^1 \times \mathbb{CP}^1$.

\begin{pro} \label{tp61}
Let $a \in \{-1,1\}$. Suppose that a minimal fan 
\begin{center}
$\{(1,0), (0,a), (-1,0), (0,-a)\}$ 
\end{center}
describes a 4-dimensional complex torus manifold $M$. Then $M$ is biholomorphic to $\mathbb{CP}^1 \times \mathbb{CP}^1$.
\end{pro}

\begin{proof}
Let $\Delta$ be the minimal fan $\{v_1=(1,0), v_2=(0,a), v_3=(-1,0), v_4=(0,-a)\}$. Suppose that $a=1$; the case $a=-1$ will be analogous. For $1 \leq i \leq 4$, let $p_i$ denote the fixed point that has weights $\{-v_{i-1},v_i\}$; here $v_0=v_4$.

Because there is a finite number of orbit types, there exists an integral element $\xi=(b_1,b_2)$ of the lie algebra $\mathfrak{t}$ of $T^2$, such that the action of the subcircle $S^1$ generated by $\xi$ on $M$ has the same fixed point set as the $T^2$-action on $M$. Then the $S^1$-weights at $p_i$ are $\{\langle -v_{i-1},\xi \rangle, \langle v_i,\xi \rangle\}$. That is, the $S^1$-weights at $p_1,p_2,p_3,p_4$ are
\begin{center}
$\{b_2,b_1\}, \{-b_1,b_2\}, \{-b_2,-b_1\}, \{b_1,-b_2\}$,
\end{center}
respectively. In particular, $b_1,b_2 \neq 0$.

By \cite[Theorem C]{CHK}, $M$ is one of the Hirzebruch surfaces $\Sigma_n$. Moreover, \cite[p. 77-78]{CHK} shows that the weights at the fixed points of an $S^1$-action on the Hirzebruch surface $\Sigma_m$ with a discrete fixed point set are of the form
\begin{center}
$\{c,d\}, \{c,-d\}, \{-c,d-mc\}, \{-c,mc-d\}$
\end{center}
for some non-zero integers $c$ and $d$ with $mc \neq d$. This implies that $n=0$, that is, $M$ is $\Sigma_0=\mathbb{CP}^1 \times \mathbb{CP}^1$. \end{proof}

Moreover, we can choose the biholomorphism in Proposition \ref{tp61} to be equivariant.

\begin{pro}
Let $a \in \{-1,1\}$. Suppose that a minimal fan 
\begin{center}
$\{(1,0), (0,a), (-1,0), (0,-a)\}$ 
\end{center}
describes a 4-dimensional complex torus manifold $M$. Then $M$ is equivariantly biholomorphic to $\mathbb{CP}^1 \times \mathbb{CP}^1$ with a $T^2$-action
\begin{center}
$(t_1,t_2) \cdot ([z_0:z_1],[y_0:y_1])=([z_0: t_1 z_1],[y_0:t_2 y_1])$
\end{center}
for all $(t_1,t_2) \in T^2 \subset \mathbb{C}^2$ and $([z_0:z_1].[y_0:y_1]) \in \mathbb{CP}^1 \times \mathbb{CP}^1$.
\end{pro}

\begin{proof}
By Proposition \ref{tp61}, $M$ is biholomorphic to $\mathbb{CP}^1 \times \mathbb{CP}^1$. Suppose that $a=1$; the case where $a=-1$ is analogous. The $T^2$-action on $M$ has 4 fixed points $p_1$, $p_2$, $p_3$, $p_4$ with weights $\{(1,0),(0,1)\}$, $\{(-1,0),(0,1)\}$, $\{(0,-1),(1,0)\}$, $\{(-1,0),(0,-1)\}$, respectively.

The weights at $p_1$ in $M$ agree with those at $([1:0],[1:0])$ in $\mathbb{CP}^1 \times \mathbb{CP}^1$ with the above action. Therefore, there is an equivariant biholomorphism between two neighborhoods $U_1$ of $p_1$ and $U_1'$ of $q_1=([1:0],[1:0])$. Similarly, there are equivariant biholomorphism between two neighborhoods $U_2$ of $p_2$ and $U_2'$ of $q_2=([0:1],[1:0])$, one between two neighborhoods $U_3$ of $p_3$ and $U_3'$ of $q_3=([1:0],[0:1])$, and one between two neighborhoods $U_4$ of $p_4$ and $U_4'$ of $q_4=([0:1],[0:1])$. Since $U_i'$ cover $\mathbb{CP}^1 \times \mathbb{CP}^1$, $U_i$ cover $M$ and the claim follows.
\end{proof}

With the above, we prove Theorem \ref{tt115}, which states that we can equivariantly blow up and down any 4-dimensional complex torus manifold to $\mathbb{CP}^1 \times \mathbb{CP}^1$ with a standard $T^2$-action with unit weights.

\begin{proof}[\textbf{Proof of Theorem \ref{tt115}}]
By Theorem \ref{tt114}, we can equivariantly blow up and down $M$ to another 4-dimensional complex torus manifold $M'$ described by a family $\Delta'$ of minimal multi-fans. By Lemma \ref{todd}, the Todd genus of $M'$ is 1. By Proposition \ref{tp63}, $\Delta'$ consists of a single minimal fan. This fan $\Delta'$ coincides with the minimal fan $\{(1,0), (0,a), (-1,0), (0,-a)\}$ for some $a \in \{-1,1\}$, up to shifting of indices of vectors in $\Delta'$. Therefore, by Proposition \ref{tp61}, $M'$ is equivariantly biholomorphic to $\mathbb{CP}^1 \times \mathbb{CP}^1$ with $T^2$-action
\begin{center}
$(t_1,t_2) \cdot ([z_0:z_1],[y_0:y_1])=([z_0: t_1 z_1],[y_0:t_2 y_1])$.
\end{center}
\end{proof}

Next, we prove Corollary \ref{tc117}, which establishes a correspondence between 4-dimensional complex torus manifolds and admissible fans.

\begin{proof}[\textbf{Proof of Corollary \ref{tc117}}] 
Let $M_1$ and $M_2$ be 4-dimensional complex torus manifolds. Suppose that $M_1$ and $M_2$ are equivariantly biholomorphic; let $f:M_1 \to M_2$ be an equivariant biholomorphism. By Lemma \ref{connected}, there exists a fan $V$ describing $M_1$. Since $f$ is an equivariant biholomorphism, $f$ takes fixed points of $M_1$ to fixed points of $M_2$, and takes $w$-spheres (Definition \ref{tsphere}) of $M_1$ to $w$-spheres of $M_2$; thus, $V$ also describes $M_2$.

Let $V$ be an admissible fan. Suppose that $V$ describes two 4-dimensional complex torus manifolds $M_1$ and $M_2$. This means that (neighborhoods of) two chains of $w$-spheres in $M_1$ and $M_2$ are equivariantly biholomorphic. Blow up and blow down are local operations; by Theorem \ref{tt115}, we can use the same blow up and blow down to each $M_i$ to $\mathbb{CP}^1 \times \mathbb{CP}^1$ with $T^2$-action
\begin{center}
$(t_1,t_2) \cdot ([z_0:z_1],[y_0:y_1])=([z_0: t_1 z_1],[y_0:t_2 y_1])$.
\end{center}
This implies that $M_1$ and $M_2$ are equivariantly biholomorphic.
\end{proof}

\begin{figure}
\centering
\begin{subfigure}[b][4.6cm][s]{.4\textwidth} 
\centering
\begin{tikzpicture}
\draw[->] (-2, 0) -- (2, 0) node[right] {$x$};
\draw[->] (0, -1.5) -- (0, 2) node[above] {$y$};
\draw[line width=0.5mm][->] (0, 0) -- (1, 0) node[below] {$v_1$};
\draw[line width=0.5mm][->] (0, 0) -- (-1, 1) node[above] {$v_2$};
\draw[line width=0.5mm][->] (0, 0) -- (0, -1) node[left] {$v_3$};
\end{tikzpicture} \qquad
\caption{Fan describing $\mathbb{CP}^2$ in Example \ref{te1}}\label{tfig10a}
\end{subfigure}
\begin{subfigure}[b][4.6cm][s]{.4\textwidth} 
\centering
\vfill
\begin{tikzpicture}[state/.style ={circle, draw}]
\node[state] (a) {$p_1$};
\node[state] (b) [above right=of a] {$p_2$};
\node[state] (c) [above left=of b] {$p_3$};
\path (a) [->] edge node [right] {$(1,0)$} (b);
\path (b) [->] edge node [right] {$(-1,1)$} (c);
\path (c) [->] edge node [left] {$(0,-1)$} (a);
\end{tikzpicture}
\caption{Graph describing $\mathbb{CP}^2$ in Example \ref{te1}}\label{tfig10b}
\end{subfigure}\qquad 
\caption{Fan and graph describing $\mathbb{CP}^2$ in Example \ref{te1}}\label{tfig10}
\end{figure}
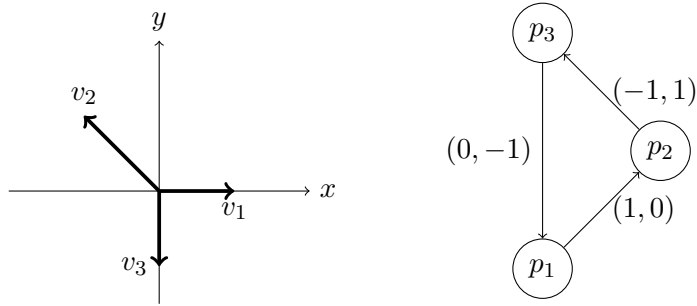

For $i=1,2$, let $M_i$ be $\mathbb{CP}^2$ or one of the Hirzebruch surfaces. While Corollary \ref{tc118} already says that any two 4-dimensional complex torus manifolds are obtained from each other by equivariant blow up and down, we shall construct explicit blow up and down from $M_1$ to $M_2$ with particular actions. We consider the following particular actions for $\mathbb{CP}^2$ and the Hirzebruch surfaces.

\begin{exa} \label{te1}
In Example \ref{te3} of the action on $\mathbb{CP}^2$, we take $v_1=(1,0)$ and $v_2=(-1,1)$. That is, let $T^2$ act on $\mathbb{CP}^2$ by
\begin{center}
$(g_1,g_2) \cdot [z_0:z_1:z_2]=[z_0:g_1 z_1: g_2z_2]$.
\end{center}
As shown in Example \ref{te3}, the 3 fixed points $[1:0:0]$, $[0:1:0]$, and $[0:0:1]$ have weights $\{(1,0), (0,1)\}$, $\{(-1,0),(-1,1)\}$, and $\{(1,-1), (0,-1)\}$, respectively, and 
the fan $V$ (Figure \ref{tfig10a}) describes $\mathbb{CP}^2$ with this action. Figure \ref{tfig10a} is the graph describing this action on $\mathbb{CP}^2$.
\end{exa}

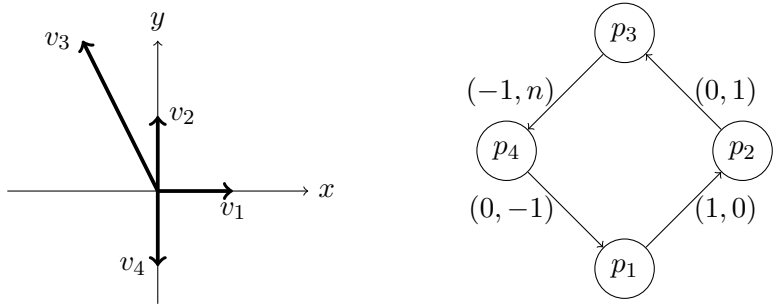
\begin{figure}
\centering
\begin{subfigure}[b][4.6cm][s]{.4\textwidth}
\centering
\begin{tikzpicture}
\draw[->] (-2, 0) -- (2, 0) node[right] {$x$};
\draw[->] (0, -1.5) -- (0, 2) node[above] {$y$};
\draw[line width=0.5mm][->] (0, 0) -- (1, 0) node[below] {$v_1$};
\draw[line width=0.5mm][->] (0, 0) -- (0, 1) node[right] {$v_2$};
\draw[line width=0.5mm][->] (0, 0) -- (-1, 2) node[left] {$v_3$};
\draw[line width=0.5mm][->] (0, 0) -- (0, -1) node[left] {$v_4$};
\end{tikzpicture}
\caption{Fan describing $\Sigma_n$ in Example \ref{te2}} \label{tfig11a}
\end{subfigure}\qquad 
\begin{subfigure}[b][4.6cm][s]{.4\textwidth}
\centering
\begin{tikzpicture}[state/.style ={circle, draw}]
\node[state] (A) {$p_1$};
\node[state] (B) [above left=of A] {$p_4$};
\node[state] (C) [above right=of A] {$p_2$};
\node[state] (D) [above right=of B] {$p_3$};
\path (B) [->] edge node[left] {$(0,-1)$} (A);
\path (A) [->] edge node [right] {$(1,0)$} (C);
\path (D) [->] edge node [left] {$(-1,n)$} (B);
\path (C) [->] edge node [right] {$(0,1)$} (D);
\end{tikzpicture}
\caption{Graph describing $\Sigma_n$ in Example \ref{te2}}\label{tfig11b}
\end{subfigure}
\caption{Fan and graph describing Hirzebruch surface $\Sigma_n$ in Example \ref{te2}}\label{tfig11}
\end{figure}

\begin{exa} \label{te2}
In Example \ref{te4}, we take $v_1=(1,0)$ and $v_2=(0,1)$. That is, let $T^2$ act on the Hirzebruch surface $\Sigma_n$ by
\begin{center}
$(g_1,g_2) \cdot ([z_0:z_1:z_2],[y_1:y_2]) = ([g_1 z_0:z_1:g_2^n z_2],[y_1:g_2 y_2])$.
\end{center}
Figure \ref{tfig11b} is the graph describing this action on $\Sigma_n$, and hence Figure \ref{tfig11a} is the fan describing this action.
\end{exa}

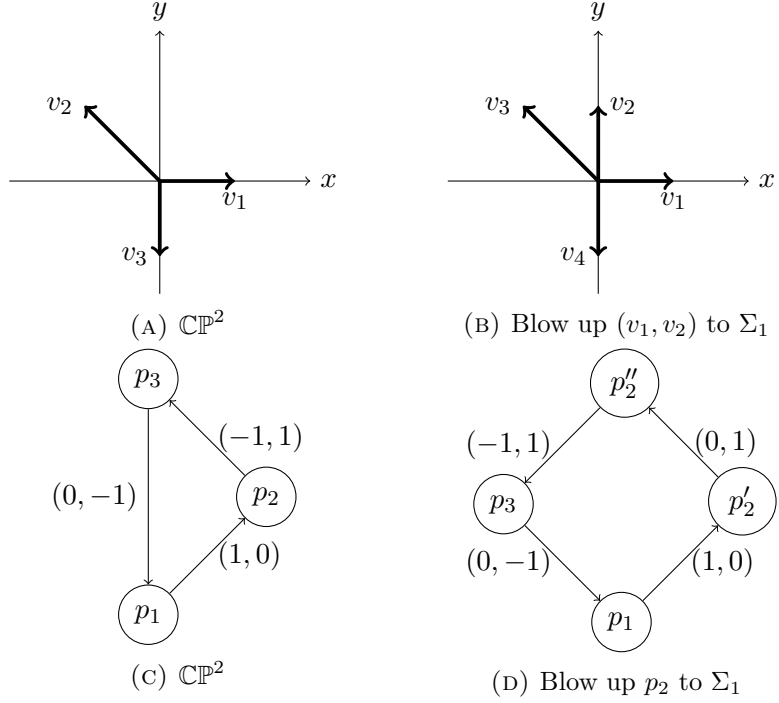
\begin{figure}
\centering
\begin{subfigure}[b][4.7cm][s]{.4\textwidth}
\centering
\begin{tikzpicture}
\draw[->] (-2, 0) -- (2, 0) node[right] {$x$};
\draw[->] (0, -1.5) -- (0, 2) node[above] {$y$};
\draw[line width=0.5mm][->] (0, 0) -- (1, 0) node[below] {$v_1$};
\draw[line width=0.5mm][->] (0, 0) -- (-1, 1) node[left] {$v_2$};
\draw[line width=0.5mm][->] (0, 0) -- (0, -1) node[left] {$v_3$};
\end{tikzpicture}
\caption{$\mathbb{CP}^2$}
\end{subfigure}\qquad
\begin{subfigure}[b][4.7cm][s]{.4\textwidth}
\centering
\begin{tikzpicture}
\draw[->] (-2, 0) -- (2, 0) node[right] {$x$};
\draw[->] (0, -1.5) -- (0, 2) node[above] {$y$};
\draw[line width=0.5mm][->] (0, 0) -- (1, 0) node[below] {$v_1$};
\draw[line width=0.5mm][->] (0, 0) -- (0, 1) node[right] {$v_2$};
\draw[line width=0.5mm][->] (0, 0) -- (-1, 1) node[left] {$v_3$};
\draw[line width=0.5mm][->] (0, 0) -- (0, -1) node[left] {$v_4$};
\end{tikzpicture}
\caption{Blow up $(v_1,v_2)$ to $\Sigma_1$}
\end{subfigure}\qquad
\begin{subfigure}[b][4.4cm][s]{.4\textwidth}
\centering
\begin{tikzpicture}[state/.style ={circle, draw}]
\node[state] (a) {$p_1$};
\node[state] (b) [above right=of a] {$p_2$};
\node[state] (c) [above left=of b] {$p_3$};
\path (a) [->] edge node [right] {$(1,0)$} (b);
\path (b) [->] edge node [right] {$(-1,1)$} (c);
\path (c) [->] edge node [left] {$(0,-1)$} (a);
\end{tikzpicture}
\caption{$\mathbb{CP}^2$}
\end{subfigure}\qquad
\begin{subfigure}[b][4.4cm][s]{.4\textwidth}
\centering
\begin{tikzpicture}[state/.style ={circle, draw}]
\node[state] (A) {$p_1$};
\node[state] (B) [above left=of A] {$p_3$};
\node[state] (C) [above right=of A] {$p_2'$};
\node[state] (D) [above right=of B] {$p_2''$};
\path (B) [->] edge node[left] {$(0,-1)$} (A);
\path (A) [->] edge node [right] {$(1,0)$} (C);
\path (D) [->] edge node [left] {$(-1,1)$} (B);
\path (C) [->] edge node [right] {$(0,1)$} (D);
\end{tikzpicture}
\caption{Blow up $p_2$ to $\Sigma_1$}
\end{subfigure}
\caption{From $\mathbb{CP}^2$ to $\Sigma_1$}\label{tfig12}
\end{figure}

\begin{figure}
\centering
\begin{subfigure}[b][4.5cm][s]{.28\textwidth}
\centering
\begin{tikzpicture}
\draw[->] (-1.5, 0) -- (1.5, 0) node[right] {$x$};
\draw[->] (0, -1.5) -- (0, 1.5) node[above] {$y$};
\draw[line width=0.5mm][->] (0, 0) -- (0.7, 0) node[below] {$v_1$};
\draw[line width=0.5mm][->] (0, 0) -- (0, 0.7) node[right] {$v_2$};
\draw[line width=0.5mm][->] (0, 0) -- (-0.7, 0.7) node[left] {$v_3$};
\draw[line width=0.5mm][->] (0, 0) -- (0, -0.7) node[left] {$v_4$};
\end{tikzpicture}
\caption{$\Sigma_n$} \label{tfig13-1}
\end{subfigure}\qquad
\begin{subfigure}[b][4.5cm][s]{.28\textwidth}
\centering
\begin{tikzpicture}
\draw[->] (-1.5, 0) -- (1.5, 0) node[right] {$x$};
\draw[->] (0, -1.5) -- (0, 1.5) node[above] {$y$};
\draw[line width=0.5mm][->] (0, 0) -- (0.7, 0) node[below] {$v_1'$};
\draw[line width=0.5mm][->] (0, 0) -- (0, 0.7) node[right] {$v_2'$};
\draw[line width=0.5mm][->] (0, 0) -- (-0.7, 1.4) node[left] {$v_3'$};
\draw[line width=0.5mm][->] (0, 0) -- (-0.7, 0.7) node[left] {$v_4'$};
\draw[line width=0.5mm][->] (0, 0) -- (0, -0.7) node[left] {$v_5'$};
\end{tikzpicture}
\caption{Blow up $(v_2,v_3)$} \label{tfig13-2}
\end{subfigure}\qquad
\begin{subfigure}[b][4.5cm][s]{.28\textwidth}
\centering
\begin{tikzpicture}
\draw[->] (-1.5, 0) -- (1.5, 0) node[right] {$x$};
\draw[->] (0, -1.5) -- (0, 1.5) node[above] {$y$};
\draw[line width=0.5mm][->] (0, 0) -- (0.7, 0) node[below] {$v_1''$};
\draw[line width=0.5mm][->] (0, 0) -- (0, 0.7) node[right] {$v_2''$};
\draw[line width=0.5mm][->] (0, 0) -- (-0.7, 1.4) node[left] {$v_3''$};
\draw[line width=0.5mm][->] (0, 0) -- (0, -0.7) node[left] {$v_4''$};
\end{tikzpicture}
\caption{Blow down $v_4'$ to $\Sigma_{n+1}$} \label{tfig13-3}
\end{subfigure}\qquad
\begin{subfigure}[b][5.1cm][s]{.4\textwidth}
\centering
\begin{tikzpicture}[state/.style ={circle, draw}]
\node[vertex] (a) at (0,0) {$p_1$};
\node[vertex] (b) at (1.7,1.7) {$p_2$};
\node[vertex] (c) at (0,3.4) {$p_3$};
\node[vertex] (d) at (-1.7,1.7) {$p_4$};
\path (a) [->] edge node[right] {$(1,0)$} (b);
\path (b) [->] edge node [right] {$(0,1)$} (c);
\path (c) [->] edge node [left] {$(-1,n)$} (d);
\path (d) [->] edge node [left] {$(0,-1)$} (a);
\end{tikzpicture}
\caption{$\Sigma_n$}\label{tfig13-4}
\end{subfigure}\qquad
\begin{subfigure}[b][5.1cm][s]{.4\textwidth}
\centering
\begin{tikzpicture}[state/.style ={circle, draw}]
\node[vertex] (a) at (2.3, 1.8) {$p_3'$};
\node[vertex] (b) at (0, 0.8) {$p_3''$};
\node[vertex] (c) at (4, 0) {$p_2$};
\node[vertex] (d) at (0, -0.8) {$p_4$};
\node[vertex] (e) at (2.3, -1.8) {$p_1$};
\path (a) [->] edge node[pos=.5, above, sloped] {$(-1,n+1)$} (b);
\path (c) [->] edge node [right] {$(0,1)$} (a);
\path (b) [->] edge node [right] {$(-1,n)$} (d);
\path (d) [->] edge node [pos=.5, above, sloped] {$(0,-1)$} (e);
\path (e) [->] edge node [right] {$(1,0)$} (c);
\end{tikzpicture}
\caption{Blow up $p_3$ to $(p_3',p_3'')$}\label{tfig13-5}
\end{subfigure}\qquad
\begin{subfigure}[b][4.6cm][s]{.5\textwidth}
\centering
\begin{tikzpicture}[state/.style ={circle, draw}]
\node[vertex] (a) at (0,0) {$p_1$};
\node[vertex] (b) at (1.7,1.7) {$p_2$};
\node[vertex] (c) at (0,3.4) {$p_3'$};
\node[vertex] (d) at (-1.7,1.7) {$p_4'$};
\path (a) [->] edge node[right] {$(1,0)$} (b);
\path (b) [->] edge node [right] {$(0,1)$} (c);
\path (c) [->] edge node [left] {$(-1,n+1)$} (d);
\path (d) [->] edge node [left] {$(0,-1)$} (a);
\end{tikzpicture}
\caption{Blow down $(p_3'',p_4)$ to $p_4'$, to $\Sigma_{n+1}$}\label{ttfig13-6}
\end{subfigure}\qquad
\caption{From $\Sigma_n$ to $\Sigma_{n+1}$}\label{tfig13}
\end{figure}

Take the $T^2$-action on $\mathbb{CP}^2$ in Example \ref{te1}. We blow up $p_2$ that has weights $\{(-1,0),(-1,1)\}$ in the sense of Lemma \ref{tl41}. Then we get the Hirzebruch surface $\Sigma_1$ with the action in Example \ref{te2} for $n=1$. Figure \ref{tfig12} illustrates this blow up on the fan level and on the graph level.

We describe how to blow up and down $\Sigma_n$ to $\Sigma_{n+1}$. Consider the $T^2$-action on $\Sigma_n$ in Example \ref{te2}. We blow up $p_3$ that has weights $\{(0,-1),(-1,n)\}$ in the sense of Lemma \ref{tl41} to $(-1,n+1)$-sphere $(p_3',p_3'')$, to get a 4-dimensional complex torus manifold $M'$ described by Figure \ref{tfig13-5}. For the fan, we blow up $(v_2,v_3)$. In $M'$, there are $(-1,n+1)$-sphere $(p_3',p_3'')$, $(-1,n)$-sphere $(p_3'',p_4)$, and $(0,-1)$-sphere $(p_4,p_1)$. Therefore, we can blow down the $(-1,n)$-sphere $(p_3'',p_4)$ in the sense of Lemma \ref{tl42} to a fixed point $p_4'$ that has weights $\{(1,-1-n),(0,-1)\}$ to get the Hirzebruch surface $\Sigma_{n+1}$ with the action in Example \ref{te2}. For the fan, we blow down $v_4'$. Figure \ref{tfig13} illustrates these blow up and blow down on the fan level and on the graph level.

\section{Further properties} \label{s5.5}

In this section, we explore properties of 4-dimensional almost complex torus manifolds. For any such manifold, we find a lower bound on the number of fixed points (Proposition \ref{tp75}), find a relationship between the Hirzebruch $\chi_y$-genus and blow up and down (Lemma \ref{tl76}), and find a necessary and sufficient condition for the Hirzebruch $\chi_y$-genus (Theorem \ref{tt78}) and that for the Chern numbers of such a manifold (Theorem \ref{tt79}).

\begin{pro} \label{tp75} Let $M$ be a 4-dimensional almost complex torus manifold. Then $M$ has at least three fixed points.
\end{pro}

\begin{proof} 
By Theorem \ref{tt19}, there exists a family of admissible multi-fans $V_1$, $\cdots$, $V_m$ describing $M$. For each $1 \leq j \leq m$, let $V_j=\{v_{j,1},\cdots,v_{j,k_j}\}$. The total number of fixed points is equal to $\sum_{j=1}^m k_j$, and we can achieve the minimum value with $m=1$ and $k_1=3$.
\end{proof}

Let $M$ be a compact almost complex manifold. The Hirzebruch $\chi_y$-genus is the genus belonging to the power series $\frac{x(1+ye^{-x(1+y)})}{1-e^{-x(1+y)}}$. Let $\dim M=2n$. We have $\chi_y(M)=\sum_{i=0}^n (\int_M T_i^n) y^i$, where $T_i^n$ is a rational combination of products of Chern classes. When $n=2$, $T_0^2=T_2^2=\frac{c_1^2+c_2}{12}$ and $T_1^2=\frac{c_1^2-5c_2}{6}$, where $c_1$ and $c_2$ denote the first and second Chern classes of $M$, respectively.

The Hirzebruch $\chi_y$-genus of a compact almost complex manifold $M$ contains three pieces of information; $\chi_{-1}(M)=\chi(M)$ is the Euler characteristic of $M$, $\chi_0(M)=\textrm{Todd}(M)$ is the Todd genus of $M$, and $\chi_1(M)=\textrm{sign}(M)$ is the signature of $M$.

For an almost complex torus manifold, all the coefficients of its Hirzebruch $\chi_y$-genus are non-zero.

\begin{theo} \label{tt71} \cite{J6}
Let $M$ be a $2n$-dimensional almost complex torus manifold. Then $a_i(M) > 0$ for $0 \leq i \leq n$, where $\chi_y(M)=\sum_{i=0}^n a_i(M) \cdot (-y)^i$ is the Hirzebruch $\chi_y$-genus of $M$. In particular, the Todd genus of $M$ is positive.
\end{theo}

For a torus action on a compact almost complex manifolds with isolated fixed points, these coefficients are non-negative.

\begin{lem} \label{tl72} \cite{J6}
Let a $k$-dimensional torus $T^k$ act on a $2n$-dimensional compact almost complex manifold $M$ with isolated fixed points. Let $\chi_y(M)=\sum_{i=0}^n a_i(M) \cdot (-y)^i$ be the Hirzebruch $\chi_y$-genus of $M$. Then the following hold:
\begin{enumerate}
\item $a_i(M) \geq 0$ for $0 \leq i \leq n$.
\item The number of fixed points is $\sum_{i=0}^n a_i(M)$.
\item For $0 \leq i \leq n$, $a_i(M)=a_{n-i}(M)$.
\end{enumerate}
\end{lem}

The following lemma is immediate from definition.

\begin{lem} \label{tl76}
Let $V$ be a multi-fan. Blow up and blow down of vectors in $V$ in the sense of Definition \ref{td112} neither increase nor decrease its winding number $T(V)$.
\end{lem}

Let $M$ be a 4-dimensional almost complex torus manifold. Blow up of a fixed point and blow down of $w$-sphere in the sense of Lemmas \ref{tl41} and \ref{tl42} neither increase nor decrease the Todd genus of $M$, while decreasing and increasing $\chi^1(M)$ by 1, respectively, where $\chi_y(M)=\sum_{i=0}^2 \chi^i(M) \cdot y^i$ is the Hirzebruch $\chi_y$-genus of $M$.

\begin{lem} \label{tl77}
Let $M$ be a 4-dimensional almost complex torus manifold. 
\begin{enumerate}[(1)]
\item Let $M'$ be an almost complex torus manifold obtained by blowing up some fixed point of $M$ in the sense of Lemma \ref{tl41}. Then the Hirzebruch $\chi_y$-genus of $M'$ is
\begin{center}
$\chi_y(M')=\chi_y(M)-y$.
\end{center}
\item Let $M''$ be an almost complex torus manifold obtained by blowing down some $w$-sphere of $M$ in the sense of Lemma \ref{tl42}. Then the Hirzebruch $\chi_y$-genus of $M''$ is
\begin{center}
$\chi_y(M'')=\chi_y(M)+y$.
\end{center}
\end{enumerate}
\end{lem}

\begin{proof}
Let $\chi_y(M)=\sum_{i=0}^2 a_i(M) \cdot (-y)^i$ denote the Hirzebruch $\chi_y$-genus of $M$ and similarly for $M'$ and $M''$. By Lemma \ref{tl72}, $a_0(M)=a_2(M)$ and the total number of fixed points of $M$ is $a_0(M)+a_1(M)+a_2(M)=2 \cdot a_0(M)+a_1(M)$, which is equal to $2 \cdot \mathrm{Todd}(M)+a_1(M)$ since $a_0(M)=\mathrm{Todd}(M)$. Similarly, $a_0(M')=a_2(M')=\mathrm{Todd}(M')$ and the total number of fixed points of $M'$ is $2 \cdot \mathrm{Todd}(M')+a_1(M')$. Since blowing up of a fixed point in the sense of Lemma \ref{tl41} increases the number of fixed points by $1$, the total number of fixed points of $M'$ is $2 \cdot \mathrm{Todd}(M)+a_1(M)+1$ and hence $2 \cdot \mathrm{Todd}(M')+a_1(M')=2 \cdot \mathrm{Todd}(M)+a_1(M)+1$. By Proposition \ref{tp43}, blow up of a fixed point in the sense of Lemma \ref{tl41} corresponds to blow up of a family of multi-fans describing $M$. On the other hand, blow up of a multi-fan $V$ does not change its winding number $T(V)$. Therefore, by Proposition \ref{tp63} and Lemma \ref{tl76}, blowing up of a fixed point in the sense of Lemma \ref{tl41} does not change the Todd genus, that is, $\mathrm{Todd}(M)=\mathrm{Todd}(M')$. Therefore, $a_1(M')=a_1(M)+1$. This proves (1). The proof of (2) is similar; since blowing down of some $w$-sphere in the sense of Lemma \ref{tl42} decreases the number of fixed points by $1$ and does not change the Todd genus, it follows that $a_1(M'')=a_1(M)-1$; thus $\chi_y(M'')=\chi_y(M)+y$. \end{proof}

We provide a necessary and sufficient condition for the Hirzebruch $\chi_y$-genus of a 4-dimensional almost complex torus manifold.

\begin{theo} \label{tt78}
There exists a 4-dimensional almost complex torus manifold $M$ with the Hirzebruch $\chi_y$-genus $\chi_y(M)=n_0-n_1 \cdot y+n_2 \cdot y^2$ if and only if $n_0=n_2$ and $n_0$, $n_1$, and $n_2$ are positive integers.
\end{theo}

\begin{proof}
Let $M$ be a 4-dimensional almost complex torus manifold. Let $\chi_y(M)=n_0-n_1 \cdot y+n_2 \cdot y^2$ be the Hirzebruch $\chi_y$-genus of $M$. By Theorem \ref{tt71}, $n_0$, $n_1$, and $n_2$ are positive. By (3) of Lemma \ref{tl72}, $n_0=n_2$.

Let $n_0$ and $n_1$ be any positive integers. First, we construct a 4-dimensional almost complex torus manifold $M'$ that has $2 \cdot n_0+1$ fixed points, where $n_0=\mathrm{Todd}(M')$. We use Theorem \ref{tt31} to construct such a manifold $M'$ by taking the following multi-fan $V=\{v_1,\cdots,v_k\}$, where $v_1=(1,0)$, $v_2=(2,1)$, $v_3=(-3,-1)$, $v_4=(4,1)$, $v_5=(-5,-1)$, $v_6=(6,1)$, $\cdots$, $v_{k-2}=(-k+2,-1)$, $v_{k-1}=(k-1,1)$, $v_{k}=(-k,-1)$; here $k=2 \cdot n_0+1$. The weights at the fixed points $p_1$, $p_2$, $\cdots$, $p_k$ are $\{(k,1),(1,0)\}$, $\{(-1,0),(2,1)\}$, $\{(-2,-1),(-3,-1)\}$, $\{(3,1),(4,1)\}$, $\{(-4,-1),(-5,-1)\}$, $\cdots$, $\{(k-2,-1),(k-1,1)\}$, $\{(-k+1,-1),(-k,-1)\}$, respectively. Since the winding number $T(V)$ of $V$ is $n_0$, by Proposition \ref{tp63}, the Todd genus of $M'$ is $n_0$. To construct a 4-dimensional almost complex torus manifold $M$ such that $\chi_y(M)=n_0-n_1 \cdot y+n_0 \cdot y^2$, take the manifold $M'$ we constructed, and blow up any fixed points in $M'$ in the sense of Lemma \ref{tl41} $(n_1-1)$-times; by Lemma \ref{tl77}, the resulting manifold $M$ satisfies $\mathrm{Todd}(M)=n_0-n_1 \cdot y+n_0 \cdot y^2$. \end{proof}

In addition, we provide a necessary and sufficient condition for the Chern numbers of a 4-dimensional almost complex torus manifold.

\begin{theo} \label{tt79}
There exists a 4-dimensional almost complex torus manifold $M$ with $c_1^2[M]=10n_0-n_1$ and $c_2[M]=2n_0+n_1$ if and only if $n_0$ and $n_1$ are positive integers.
\end{theo}

\begin{proof}
Let $M$ be a 4-dimensional almost complex torus manifold. By Theorem \ref{tt78}, the Hirzebruch $\chi_y$-genus of $M$ is
\begin{center}
$\chi_y(M)=n_0-n_1 \cdot y+n_0 \cdot y^2$
\end{center}
for some positive integers $n_0$ and $n_1$. On the other hand,
\begin{center}
$\displaystyle \chi_y(M)=\sum_{i=0}^2 \left(\int_M T_i^2 \right) \cdot y^i$,
\end{center}
where $T_0^2=T_2^2=\frac{c_1^2+c_2}{12}$ and $T_1^2=\frac{c_1^2-5c_2}{6}$. Hence $\int_M c_1^2=10 \cdot n_0-n_1$ and $\int_M c_2=2 \cdot n_0+n_1$.

Let $n_0$ and $n_1$ be positive integers. By Theorem \ref{tt78}, there exists a 4-dimensional almost complex manifold $M$ with $\chi_y(M)=n_0 - n_1 \cdot y + n_0 \cdot y^2$. Since $n_0=\int_M \frac{c_1^2+c_2}{12}$ and $n_1=\int_M \frac{c_1^2-5c_2}{6}$, the second claim of this theorem holds. \end{proof}

Note that the results for circle actions analogous to Theorems \ref{tt78} and \ref{tt79} were obtained in \cite{J5}.

\section{Few fixed points} \label{s5.6}

In this section, we classify 4-dimensional almost complex torus manifolds with few fixed points. By Proposition \ref{tp75}, any 4-dimensional almost complex torus manifold has at least 3 fixed points. If there are 3 fixed points, a fan (and a graph) describing the manifold agrees with that describing a linear action on $\mathbb{CP}^2$ in Example \ref{te3}.

\begin{theo}
Let $M$ be a 4-dimensional almost complex torus manifold with 3 fixed points. Then a fan $\{v_1,v_2,-v_1-v_2\}$ describes $M$, for some $v_1$ and $v_2$ that form a basis of $\mathbb{Z}^2$. Alternatively, Figure \ref{tfig14} describes $M$. In particular, the weights at the fixed points are $\{v_1+v_2,v_1\}$, $\{-v_1,v_2\}$, and $\{-v_2,-v_1-v_2\}$.
\end{theo}

\begin{proof}
Since there are 3 fixed points, by Theorem \ref{tt19}, an admissible fan $V=\{v_1,v_2,v_3\}$ describes $M$ for some $v_1,v_2,v_3 \in \mathbb{Z}^2$. Since $V$ is admissible, $v_{i-1}$ and $v_i$ form a basis of $\mathbb{Z}^2$, and $v_3=-a_2 v_2-v_1$ and $v_2=-a_1 v_1-v_3$ for some integers $a_i$. Then $v_2=-a_1 v_1-v_3=-a_1 v_1- (-a_2 v_2-v_1)=-a_1 v_1+v_1+a_2 v_2=(1-a_1)v_1+a_2 v_2$. Since $v_1$ and $v_2$ form a basis of $\mathbb{Z}^2$, this implies that $a_1=1$ and $a_2=1$. Hence $v_3=-v_1-v_2$. \end{proof}

If there are 4 fixed points, a fan (and a graph) describing the manifold agrees with the one describing the action on the Hirzebruch surface $\Sigma_n$ in Example \ref{te4}.

\begin{theo}
Let $M$ be a 4-dimensional almost complex torus manifold with 4 fixed points. Then a fan $\{v_1,v_2,-v_1+av_2,-v_2\}$ describes $M$ for some $v_1$ and $v_2$ that form a basis of $\mathbb{Z}^2$ and for some integer $a$. Alternatively, Figure \ref{tfig15} describes $M$. In particular, the weights at the fixed points are $\{v_2,v_1\}$, $\{-v_1,v_2\}$, $\{-v_2,-v_1+av_2\}$, $\{v_1-av_2,-v_2\}$.
\end{theo}

\begin{proof}
Since there are 4 fixed points, by Theorem \ref{tt19}, an admissible fan $V=\{w_1,w_2,w_3,w_4\}$ describes $M$ for some $w_1,w_2,w_3,w_4 \in \mathbb{Z}^2$. Since $V$ is admissible, $w_{i-1}$ and $w_i$ form a basis of $\mathbb{Z}^2$, and $w_4=-a_3 w_3-w_2$, $w_3=-a_2 w_2-w_1$, $w_2=-a_1 w_1-w_4$, and $w_1=-a_4 w_4-w_3$, for some integers $a_i$. Then
\begin{center}
$w_2=-a_1 w_1-w_4=-a_1 w_1-(a_3 w_3-w_2)=-a_1 w_1-a_3 w_3+w_2=-a_1 w_1- a_3 (-a_2 w_2-w_1)+w_2=(a_3-a_1)w_1+(a_2 a_3 +1)w_2$. 
\end{center}
Since $w_1$ and $w_2$ form a basis of $\mathbb{Z}^2$, this implies that $a_2=0$ or $a_3=0$.

If $a_2=0$, then $w_3=-w_1$ and the theorem follows if we take $v_1=w_4$, $v_2=w_1$, and $a=-a_1$.

If $a_3=0$, then $w_4=-w_2$ and the theorem follows if we take $v_1=w_1$, $v_2=w_2$, and $a=-a_2$.
\end{proof}

\end{document}